\documentclass[twoside,a4paper,12pt,reqno]{amsart}
\usepackage{amsthm}
\usepackage{orcidlink}

\usepackage[a4paper,inner=2cm,outer=2cm,top=3cm,bottom=3cm]{geometry}
\usepackage[dvipsnames]{xcolor}
\usepackage{amsthm,amsmath,amsfonts,graphicx,geometry,lipsum,amsxtra}
\usepackage{hyperref}
\usepackage{mathrsfs}
\usepackage{verbatim}
\usepackage{tikz}
\usepackage{epigraph}

\numberwithin{equation}{section}
\usepackage{esint}
\usepackage{mathtools}

\usepackage{etoolbox}
\usepackage{mathtools}
\hypersetup{
    colorlinks=true,
      linkcolor=blue,
    citecolor=red,
    urlcolor=red,
    pdftitle={[BKP23] Determination of lower order perturbations of a polyharmonic operator in two dimensions},
    }
\usepackage{bbm}   
\usepackage{comment}
\usepackage{amsmath}
\usepackage{tikz}
\usepackage{epigraph}
\usepackage{lipsum}
\allowdisplaybreaks

\definecolor{titlepagecolor}{cmyk}{1,.60,0,.40}
\DeclareFixedFont{\titlefont}{T1}{ppl}{b}{it}{0.5in}

\def\th@plain{%
  \thm@notefont{}
  \itshape 
}
\def\th@definition{%
  \thm@notefont{}
  \normalfont 
}
\makeatother

\usepackage{srcltx} 
\usepackage{amscd}
\usepackage{amssymb}
\usepackage{amsmath}
\usepackage{pb-diagram}
\usepackage{graphicx}
\usepackage{hyperref}
\usepackage{array}
\usepackage[all]{xy}
\xyoption{matrix}
\xyoption{arrow}
\usepackage{pdfsync}
\usepackage{physics}
\usepackage{enumitem}
 
\usepackage{thmtools}
\usepackage{amssymb}

\DeclareUnicodeCharacter{2212}{-,+}
\usepackage[none]{hyphenat}
\usepackage{bm}
\usepackage{amsthm} 

\theoremstyle{plain}
\newtheorem{theorem}{Theorem}[section]
\newtheorem{lemma}[theorem]{Lemma}

\theoremstyle{definition}

\usepackage{cancel}
\usepackage{ulem} 

\calclayout

\begin{document}
\title[The Wave equation]{Inverse
problems for a coupled system of wave equations with point source-receiver data}

\author[R. Bhardwaj]{Rahul Bhardwaj\,\orcidlink{0009-0007-5122-7781}}
\address{Rahul Bhardwaj\\ Department of Mathematics, Indian Institute of Technology, Ropar, Rupnagar-140001, Punjab, India}
\email{bhardwaj161067@gmail.com, rahul.24maz0014@iitrpr.ac.in}

\author[M. Vashisth]{Manmohan Vashisth\,\orcidlink{0000-0002-3417-4055}}
\address{Manmohan Vashisth\\Department of Mathematics, Indian Institute of Technology, Ropar, Rupnagar-140001, Punjab, India.}
\email{manmohanvashisth@iitrpr.ac.in}

\vspace*{-1cm}
\begin{abstract}
The present manuscript consists of inverse problems for a coupled system of wave equations with potential in $\mathbb{R}^3$. By establishing the fundamental solution to the aforementioned operator, we study the uniqueness aspects of the inverse problem of recovering the matrix-valued potential coefficient from time-dependent measurements. We consider these inverse problems in two different cases: 
(i) the {\it coincident} setup, where the source and receiver are located at a single point, and (ii) the {\it non-coincidence or separated} setup, in which case source and receiver are situated at distinct locations. The problems considered here are under-determined; hence, some additional assumptions for the potential are expected to guarantee the uniqueness of the inverse problems considered in this article. We proved the desired uniqueness results under some extra assumptions on the coefficients. 
\medskip
		
		
\end{abstract}
\maketitle
\section{Introduction and statement of main results}
\subsection{Problem of interest}  We begin with considering  a coupled  system of  wave equations perturbed with a matrix-valued potential $((\beta_{ij}))_{1\leq i,j\leq 2}$ and a source located at $O:=(0,0,0)$ by 
\begin{align}\label{coupled equation}
    \begin{cases}
       \left( \Box - \beta_{11}(x)\right)U_{1}(t,x) - \beta_{12}(x)U_{2}(t,x) = \delta(t,x),&\quad(t,x)\in \mathbb{R} \times\mathbb{R}^3, \\
       \left( \Box - \beta_{22}(x)\right)U_{2}(t,x) - \beta_{21}(x)U_{1}(t,x) = \delta(t,x), &\quad(t,x)\in \mathbb{R} \times\mathbb{R}^3, \\
       U_{1}(t,x) = U_{2}(t,x) =0, &\quad(t,x)\in (-\infty, 0)\times\mathbb{R}^3.
    \end{cases}
\end{align}
 In Equation \eqref{coupled equation}, each real-valued functions $\beta_{ij}$, $1\leq i,j \leq 2$, is an infinitely differentiable  function defined on $\mathbb{R}^3$, and $\delta$ denotes the Dirac delta distribution  concentrated at $(O,0)$. If  the  matrix-valued potential $\mathfrak{P}$, the displacement vector  $\overrightarrow{U}(t,x)$ and $\overrightarrow{\mathscr{A}}$  are given by 
\begin{align}\label{notations}
   \mathfrak{P}(x) :=  \begin{bmatrix}
       \beta_{11}(x) &  \beta_{12}(x)\\
       \beta_{21}(x) &  \beta_{22}(x)
     \end{bmatrix},
    \quad \overrightarrow{U}(t,x):= \begin{bmatrix}
        U_1(t,x)\\
        U_2(t,x)
    \end{bmatrix}
    \quad \text{ and } \quad \overrightarrow{\mathscr{A}} := 
    \begin{bmatrix}
        1\\1
    \end{bmatrix},
\end{align}
then, the system of equations in \eqref{coupled equation} can be expressed  as follows 
\begin{align}\label{system of IVP}
    \begin{cases}
        (\Box I_{2\times2} -  \mathfrak{P}(x))\overrightarrow{U}(t,x)  = {\delta}(t,x)\overrightarrow{\mathscr{A}},&\quad(t,x)\in \mathbb{R} \times\mathbb{R}^3, \\
       \overrightarrow {U}(t,x)  = \overrightarrow{0}, &\quad (t,x)\in (-\infty, 0)\times\mathbb{R}^3, 
    \end{cases}
\end{align}
where $I_{2\times2}$ stands for an identity matrix of order $2\times2$ and $\overrightarrow{0}:=\begin{bmatrix}
    0\\0
\end{bmatrix}$. 
From a mathematical perspective, the operator introduced in \eqref{coupled equation} can be viewed as a perturbation of the classical D'Alembert operator, usually denoted by $\Box := \partial_t^2 - \Delta_x$. This operator is fundamental to the study of wave propagation, naturally appearing in a class of hyperbolic partial differential equations (PDEs) and reflecting fundamental physical principles, including causality and a finite propagation speed.

 The current manuscript aims to consider several unique determination results related to the inverse problem of recovering 
 the  potential matrix
$\mathfrak{P}$ from the given information (to be specified below) of the solution to \eqref{coupled equation}. Before proceeding further to the formulation of the inverse problem, we first mention a result related to the direct problem associate  to  the
  point-source problem described by Equation 
\eqref{system of IVP}. In particular, we  show (see section \ref{well-posedness} for more details) that Equation \eqref{system of IVP} admits   a unique solution  $\overrightarrow{U}$,  which can be expressed as follows
\begin{align}\label{Expression for u vector}
    \overrightarrow{U}(t,x) = \frac{\delta(t-|x|)}{4\pi|x|}\overrightarrow{\mathscr{A}} + \mathscr{H}(t-|x|)\overrightarrow{R}(t,x),
\end{align}
where $\mathscr{H}$ denotes the Heaviside distribution, and the second term on the right-hand side of  \eqref{Expression for u vector}, will survive in space-time region given by $\{(t,x)\in\mathbb{R}^3\times\mathbb{R}:\ t>|x|\}$, and in this region $\overrightarrow{R}:=\begin{bmatrix}
    R_1\\
    R_2
\end{bmatrix}$  is a solution to the following  system of boundary value problem (BVP)  
\begin{align}\label{Goursat problem}
\begin{aligned}
\begin{cases}
   \left(\Box I_{2\times2} -  \mathfrak{P}(x)\right)\overrightarrow{R}(t,x) = \overrightarrow{0},\quad  x\in \mathbb{R}^3, t> |x|, \\
    \overrightarrow{R}(x,|x|) =   \begin{bmatrix}
       \frac{1}{8\pi}\int_0^1 (\beta_{11} + \beta_{12})(sx) ds\\
      \frac{1}{8\pi}  \int_0^1 (\beta_{21} + \beta_{22})(sx) ds
    \end{bmatrix}.
    \end{cases} 
    \end{aligned} 
\end{align} 
 The above BVP is known as the Goursat problem in the literature (see \cite{Rakesh2008,blaasten2017well,w1} and references therein). Following the techniques used in \cite{blaasten2017well,Friedlander_1975}, we establish the well-posedness of the BVP given by Equation \eqref{Goursat problem} in section \ref{well-posedness}.

In the present article, we
study the inverse problems of determining  the matrix-potential $\mathfrak{P}(x)$, from   $\overrightarrow{U}(t,a)$ where  $a\in \mathbb{R}^3$ is a fixed point and   $0\leq t\leq T$, for some finite time  $T>0$ to be specified later where $\overrightarrow{U}$ is a solution the IVP governed by \eqref{coupled equation}. If  $a=O$,  then the inverse problem under consideration is known as the {\it coincident} source-receiver case, while $a\neq O$ corresponds to the {\it non-coincidence or separated}  source-receiver case. We refer to \cite{Rakesh2008,vashisth2025unique,w1} for more details on it. 
The inverse problem for determining the matrix potential $\mathfrak{P}(x)$ poses significant challenges because of the fact that the measured data depends on one dimension while the unknown coefficient $\mathfrak{P}(x)$ is a function of three spatial variables. Therefore, it is natural to expect some structural assumptions on the components of the matrix potential $\mathfrak{P}(x)$ in order to establish the unique recovery. In this article, we established the unique determination of $\mathfrak{P}(x)$ under the following assumptions on coefficients: (i) when the coefficients are comparable in the sense $\beta^{(1)}_{ij}(x) \geq \beta^{(2)}_{ij}(x) $, $1\leq i,j\leq 2$ or (ii) when the coefficients possess a certain symmetry, which is radial for the coincident case and an ellipsoidal symmetry for the separated source-receiver pair case. Please refer to Theorems \ref{thm1}, \ref{thm2}, \ref{thm3} and \ref{thm4} for more details. 
\subsection{Physical significance}
Waves occur widely in nature, and they are fundamental to many phenomena, including light, sound, earthquakes, fluid surface waves, electromagnetic radiation, and many others. Inverse problems associated with the wave equation, which is part of the class of hyperbolic PDEs, have been widely investigated in the literature due to their relevance to various physical applications, such as geophysics, medical imaging, and non-destructive testing. In these settings, waves are emitted into a medium, and the response is measured to infer internal properties of the medium, such as density or stiffness parameters. One classical configuration involves using a point-source to generate the wave and recording its response at a point receiver, either at the same location or at a distinct location. Various mathematical models are devised for it. Here we consider a particular mathematical model inspired by its applications in geophysics; see \cite{symes2009seismic} and references therein for more details. In this setting, we consider the coupled wave system, a simplified mathematical model for the propagation of interacting wave modes in a heterogeneous medium. The diagonal coefficients represent intrinsic medium properties that affect each wave component individually, while the off-diagonal terms account for coupling between the components. Such coupling naturally arises in various applications, such as seismic wave propagation, where different wave modes interact due to anisotropy, layering, or elastic heterogeneities in the subsurface. A point-source represents a controlled excitation, such as an explosion or an acoustic pulse, that simultaneously generates multiple wave components. This model keeps the main physical effects while still being easy to analyze for inverse problems.
\subsection{Related articles}
 We briefly mention some related works on the problem studied in the present article.  The inverse problem related to determining the $2\times 2$ matrix coefficients appears in a system of first-order hyperbolic PDEs from one-dimensional data was studied by Bube and Burridge (see \cite{bube1983one}). They proved that the reconstruction of the coefficients can be closely related to the Cholesky factorization of specific matrices derived using the measured data, both in the continuum setting and in the associated discrete formulation.  In \cite{sacks1996impedance}, Rakesh and Sacks investigate the inverse problem for recovering the unknown impedance coefficient associated with a one-dimensional second-order hyperbolic equation using the transmission data measured at a specific depth over a finite time interval. Romanov in \cite{Romanov1992} addressed the identification of damping and potential coefficients that remain constant outside a bounded, simply connected region in $\mathbb{R}^3$, and approached the problem via a reduction to integral geometry. In \cite{Rakesh2008}, Rakesh considered the inverse problem for the aspects of uniqueness where the source and receiver coincide, focusing on radially symmetric or comparable coefficients. In \cite{w1}, Vashisth considers the same problem for source and receiver data located at distinct points. Extensions to more general settings were considered in \cite{RakeshSacks2011, uhlmann2014point}, where angular control conditions on the coefficients were introduced and considered the inverse problem for the determination of the radially symmetric potential uniquely,
when receiver data at the whole boundary of the unit disk in $\mathbb{R}^3$. In \cite{Rakesh1998, Stefanov1990, uhlmann2014point}, authors explored inverse back-scattering problems under varying data configurations. We also refer to \cite{rakesh2010stability}, where Rakesh and Sacks study a system of hyperbolic PDEs with two distinct propagation speeds and diagonal damping and potential matrices, and address the inverse problem of stable recovery of these matrices from initial data and the impulse source at the boundary.  In \cite{blaasten2017well}, Bl\aa{}sten analyzed the well-posedness of both the point-source and Goursat problems. Additionally, they studied the inverse back-scattering problem associated with the point sources, focusing on stability estimates under the assumption of angularly controlled potentials.
Furthermore, studies in \cite{KRISHNAN2023622} aim to determine the time-dependent lower-order coefficients associated with the three-dimensional wave operator from the knowledge of point-source measurements, with an emphasis on stability analysis. Inverse problems related to the hyperbolic equation are widely studied, particularly for the context of determining medium properties using point-source data, see \cite{rakesh1993inverse,  rakesh2003inverse,klibanov2005some, li2006estimation,romanov2013integral,  vashisth2025unique} and references therein. Determination of the potential associated with the wave equation using boundary measurements has been investigated from various perspectives; see \cite{ kian2017unique, eskin1997inverse, avdonin1992boundary, bellassoued2019stability} and for the matrix-valued potential, see \cite{Mishra10122021, ben2025stable, Kumar2024StableDO, khanfer2019inverse,filippas2025recoverymatrixvaluedpotential} and references therein. Motivated by these works, our aim is to extend the works \cite{w1,Rakesh2008} in the context of the uniqueness of matrix potential from the point-source problem \eqref{system of IVP}.

\subsection{Main results} \label{Subsec: main result}
As mentioned before, our aim in this article is to study inverse problems for the unique recovery of the coefficients appearing in a coupled system of wave equations from information about the solutions measured at a fixed point over a finite time interval. We establish uniqueness results for the aforementioned inverse problems, either when the coefficients are comparable (see below, Theorems \ref{thm1} and \ref{thm3}) or when they belong to the following admissible sets of matrix-valued potentials. 
\begin{align*}
\mathcal{A}_1
&:=\left\{\mathfrak{P}(x)=\begin{bmatrix}
       \beta_{11}(x) &  \beta_{12}(x)\\
       \beta_{21}(x) &  \beta_{22}(x)
    \end{bmatrix} \;:\; \beta_{11}(x)=\beta_{12}(x)=\beta_{21}(x)=\beta_{22}(x) \right\},\\
\mathcal{A}_2
&:= \left\{
\mathfrak{P}(x)=
\begin{bmatrix}
\beta(x) & \beta_{12}(x)\\
\beta_{21}(x) & \beta(x)
\end{bmatrix}
\;:\;
\beta_{12},\,\beta_{21}\ \text{are prescribed, while } \beta  \ \text{is unknown}
\right\},\\
\mathcal{A}_3
&:= \left\{
\mathfrak{P}(x)=
\begin{bmatrix}
\beta_{11}(x) & \beta(x)\\
\beta(x) & \beta_{22}(x)
\end{bmatrix}
\;:\;
\beta_{11},\,\beta_{22}\ \text{are prescribed, while } \beta \ \text{is unknown}
\right\}.
\end{align*}
For each of the admissible classes $\mathcal{A}_p$, $p=1,2,3$, mentioned above, the uniqueness results are established 
under some additional assumptions on the coefficients which is radial symmetry for the coincident case (see Theorem \ref{thm2} below) and 
an ellipsoidal symmetry when data is given by a separated source-receiver pair (see Theorem \ref{thm4} below). More precisely, Theorems \ref{thm1}-\ref{thm4}, stated below, are the main results of the present article.
\begin{theorem}\label{thm1} Let $\overrightarrow{\mathscr{A}}$ be as in Equation \eqref{notations}.
For $k=1,2$,  suppose $\mathfrak{P}^{(k)}$ be  $2\times2$ matrix-valued functions on $\mathbb{R}^3$ having entries
$\beta^{(k)}_{ij}\in C^{\infty}(\mathbb{R}^3)$ for $1\le i,j\le 2$ and $\overrightarrow{U}^{(k)}$ solve the following  IVP 
  \begin{align}\label{e1}
     \begin{cases}
      \left(\Box I_{2\times2} -  \mathfrak{P}^{(k)}(x)\right)\overrightarrow{U}^{(k)}(t,x)  = {\delta}(t,x)\overrightarrow{\mathscr{A}},\quad&(t,x)\in \mathbb{R} \times\mathbb{R}^3, \\
       \overrightarrow{U}^{(k)}(t,x)  = \overrightarrow{0},  &(t,x)\in (-\infty, 0)\times\mathbb{R}^3.
  \end{cases}
  \end{align}
Now if 
$\beta^{(1)}_{ij}(x)\ge \beta^{(2)}_{ij}(x),$ for each  $x\in\mathbb{R}^3,\; 1\le i,j\le 2$
and 
$\overrightarrow{U}^{(1)}(t,0)=\overrightarrow{U}^{(2)}(t,0),$ for each $t\in(0,T]$, then 
\begin{align*}
    \mathfrak{P}^{(1)}(x)=\mathfrak{P}^{(2)}(x), \quad \, \mbox{for each}\  x\in\mathbb{R}^3 \ \text{such that}\ |x|\le \frac{T}{2}.
  \end{align*}
\end{theorem}
    
\begin{theorem}\label{thm2}
 Let $\overrightarrow{\mathscr{A}}$ be as in Equation \eqref{notations}. For  $k=1,2$, let $\mathfrak{P}^{(k)}\in \mathcal{A}_p$, for any $p=1,2,3$, having entries 
$\beta^{(k)}_{ij}\in C^{\infty}(\mathbb{R}^3)$ for $1\leq i,j\leq 2$ and $\overrightarrow{U}^{(k)}$ solve the following IVP 
  \begin{align}\label{e2}
       \begin{cases}
      \left(\Box I_{2\times2} -  \mathfrak{P}^{(k)}(x)\right)\overrightarrow{U}^{(k)}(t,x)  = \delta(t,x)\overrightarrow{\mathscr{A}} ,\quad&(t,x)\in \mathbb{R} \times\mathbb{R}^3, \\
       \overrightarrow{U}^{(k)}(t,x)  = \overrightarrow{0},  &(t,x)\in (-\infty, 0)\times\mathbb{R}^3. 
  \end{cases} 
  \end{align}
  Assume further that $\beta^{(k)}_{ij}(x)$ = $b^{(k)}_{ij}(|x|)$, for each $x\in \mathbb{R}^3$, $1\leq i,j\leq 2$ and $k=1,2$. Now if 
$\overrightarrow{U}^{(1)}(t,0)=\overrightarrow{U}^{(2)}(t,0),$ for each $t\in(0,T]$, then 
\begin{align*}
    \mathfrak{P}^{(1)}(x)=\mathfrak{P}^{(2)}(x), \quad \mbox{for each} \  x\in\mathbb{R}^3 \ \text{such that}\ |x|\le \frac{T}{2}.
  \end{align*}
\end{theorem}

\begin{theorem}\label{thm3}
 Let $\overrightarrow{\mathscr{A}}$ be as in Equation \eqref{notations}.
For $k=1,2$,  suppose $\mathfrak{P}^{(k)}$ be  $2\times2$ matrix-valued functions on $\mathbb{R}^3$ having entries
$\beta^{(k)}_{ij}\in C^{\infty}(\mathbb{R}^3)$ for $1\le i,j\le 2$ and $\overrightarrow{U}^{(k)}$ solve the following  IVP  
\begin{align}\label{e3}
\begin{cases}
\left(\Box I_{2\times2}-\mathfrak{P}^{(k)}(x)\right)\overrightarrow{U}^{(k)}(t,x)
= {\delta}(t,x)\overrightarrow{\mathscr{A}}, \quad &(t,x)\in \mathbb{R}\times\mathbb{R}^3,\\[2mm]
\overrightarrow{U}^{(k)}(t,x)=\overrightarrow{0},  &(t,x)\in (-\infty, 0)\times\mathbb{R}^3.
\end{cases}
\end{align}
Now if 
$\beta^{(1)}_{ij}(x)\ge \beta^{(2)}_{ij}(x),$ for each  $x\in\mathbb{R}^3,\; 1\le i,j\le 2$, and 
$\overrightarrow{U}^{(1)}(t,e)=\overrightarrow{U}^{(2)}(t,e),$ for each $t\in(0,T],$ where $T>1$ and $e = (1,0,0)$ be the fixed unit vector, then 
\begin{align*}
    \mathfrak{P}^{(1)}(x)=\mathfrak{P}^{(2)}(x), \quad \mbox{for each} \ \, x\in\mathbb{R}^3 \ \text{such that}\ \ |x|+|x-e|\le T.
\end{align*}
\end{theorem}

\begin{theorem}\label{thm4}
  Let $\overrightarrow{\mathscr{A}}$ be as in Equation \eqref{notations}. For  $k=1,2$, let $\mathfrak{P}^{(k)}\in \mathcal{A}_p$, for any $p=1,2,3$, having entries 
$\beta^{(k)}_{ij}\in C^{\infty}(\mathbb{R}^3)$ for $1\leq i,j\leq 2$ and $\overrightarrow{U}^{(k)}$ solve the following IVP
   \begin{align}\label{e4}
  \begin{cases}
      \left(\Box I_{2\times2} -  \mathfrak{P}^{(k)}(x)\right)\overrightarrow{U}^{(k)}(t,x)  = {\delta}(t,x)\overrightarrow{\mathscr{A}},\quad&(t,x)\in \mathbb{R} \times\mathbb{R}^3, \\
       \overrightarrow{U}^{(k)}(t,x)  = \overrightarrow{0},  &(t,x)\in (-\infty, 0)\times\mathbb{R}^3. 
  \end{cases}
  \end{align}
 Assume further that $\beta^{(k)}_{ij}(x)$ = $b^{(k)}_{ij}(|x|+|x-e|)$ for each $x\in \mathbb{R}^3$, $1\leq i,j\leq 2$, and $k=1,2$. Now if 
$\overrightarrow{U}^{(1)}(t,e)=\overrightarrow{U}^{(2)}(t,e),$ for each $t\in(0,T],$ where $T>1$ and $e = (1,0,0)$ be the fixed unit vector, then 
\begin{align*}
    \mathfrak{P}^{(1)}(x)=\mathfrak{P}^{(2)}(x), \quad \mbox{for each}
     \, \  x\in\mathbb{R}^3 \ \text{such that} \ \ |x|+|x-e|\le T.
\end{align*}
\end{theorem}

\noindent Our approach for proving Theorems \ref{thm1}, \ref{thm2}, \ref{thm3}, and \ref{thm4} is based on the construction of the integral identity, which is derived from the solution of the adjoint problem. This is followed by the use of spheroidal and prolate spheroidal coordinates to obtain an integral inequality, which, together with Gr\"onwall's inequality, yields the desired uniqueness result. This work contributes to the understanding of inverse problems with minimal measurement data and highlights conditions under which unique recovery is possible. This work can be considered an extension of prior works \cite{Rakesh2008,w1,blaasten2017well}, which addresses the aforementioned problems for a single wave equation with potential.

\subsection{Organization of the article}
The structure of the remainder of this article is given as follows. In Section~\ref{well-posedness}, we establish the Fundamental solution to the point-source problem and the well-posedness of the Goursat BVP given by Equation \eqref{Goursat problem}. Section \ref{sec_proofs_main_results}, which contains the proofs of the main results of this article, is split into two subsections: \ref{Sec; proofs comparable results} and \ref{Sec; proofs radial results}. In Subsection~\ref{Sec; proofs comparable results}, we present the proofs of Theorems \ref{thm1} and \ref{thm3}, which concern comparable coefficients. The proofs of Theorems \ref{thm2} and \ref{thm4}, which are considered under the symmetry assumptions, are provided in the Subsection~\ref{Sec; proofs radial results}.

\section{Fundamental solution}\label{well-posedness}
This section is devoted to deriving the fundamental solutions of the coupled system of wave equations with a point-source and to proving the well-posedness of the associated Goursat problem. It is organized into two subsections: the first addresses the case of a source at the origin, and the second presents the corresponding solution when the source is located at $e$.
\subsection{The point-source at the origin}
As mentioned above, in this subsection, we derive the fundamental solutions to the coupled system of wave equations with a point-source situated at the origin and prove the well-posedness for the associated Goursat problem.
We start with observing  (see, for instance, \cite{Friedlander_1975}) that the fundamental solution of the wave operator satisfies 
\begin{align}\label{Green's function}
    (\partial_t^2 - \Delta)\frac{\delta(t-|x|)}{4\pi|x|} = \delta(t,x).
\end{align}
Consequently, we have
\begin{align*}
    (\Box I_{2\times2})\frac{\delta(t-|x|)}{4\pi|x|} \overrightarrow{\mathscr{A}}= {\delta}(t,x)\overrightarrow{\mathscr{A}}.
\end{align*}
In the present setting, the governing equation in the point-source problem \eqref{system of IVP} involves a zeroth-order perturbation; therefore, motivated by \cite{Friedlander_1975,Rakesh2008,blaasten2017well}, we look for the  following ansatz for the solution associated with the
operator $\Box I_{2\times2}-\mathfrak{P}(x)$ 
\begin{align}\label{Ex for u vector}
    \overrightarrow{U}(t,x) = \frac{\delta(t-|x|)}{4\pi|x|}\overrightarrow{\mathscr{A}} + \mathscr{H}(t-|x|)\overrightarrow{R}(t,x),
\end{align}
where
\begin{align*}
    \overrightarrow{\mathscr{A}} = 
    \begin{bmatrix}
        1\\1
    \end{bmatrix},\qquad  \overrightarrow{R} = 
    \begin{bmatrix}
        R_1\\R_2
    \end{bmatrix}
\end{align*} and 
$\mathscr{H}$ denote the Heaviside distribution.
Applying  the operator $\Box I_{2\times2}-\mathfrak{P}(x)$ to the ansatz, gives us 
\begin{align}\label{eq:A+B+C+D}
    &\left(\Box I_{2\times2} - \mathfrak{P}(x)\right)\overrightarrow{U}(t,x) = 
    \begin{bmatrix}
        \partial_t^2 - \Delta - \beta_{11}(x) & -\beta_{12}(x) \\
        -\beta_{21}(x) & \partial_t^2 - \Delta - \beta_{22}(x)
    \end{bmatrix}
    \begin{bmatrix}
        U_1(t,x) \\ U_2(t,x)
    \end{bmatrix} \nonumber\\
    &= 
    \begin{bmatrix}
        \partial_t^2 - \Delta - \beta_{11}(x) & -\beta_{12}(x) \\
        -\beta_{21}(x) & \partial_t^2 - \Delta - \beta_{22}(x)
    \end{bmatrix}
    \left\{
        \frac{\delta(t - |x|)}{4\pi|x|}
        \begin{bmatrix}
            1 \\ 1
        \end{bmatrix}
        + \mathscr{H}(t - |x|)
        \begin{bmatrix}
            R_1(t,x) \\ R_2(t,x)
        \end{bmatrix}
    \right\} \nonumber\\
    &=
    \begin{bmatrix}
        \left(\partial_t^2 - \Delta\right)\dfrac{\delta(t - |x|)}{4\pi|x|} \\
        \left(\partial_t^2 - \Delta\right)\dfrac{\delta(t - |x|)}{4\pi|x|}
    \end{bmatrix}
    - 
    \begin{bmatrix}
        \left(\beta_{11}(x) + \beta_{12}(x)\right)\dfrac{\delta(t - |x|)}{4\pi|x|} \\
        \left(\beta_{21}(x) + \beta_{22}(x)\right)\dfrac{\delta(t - |x|)}{4\pi|x|}
    \end{bmatrix} \nonumber\\
    &\quad +
    \begin{bmatrix}
        \left(\partial_t^2 - \Delta\right)\bigl[\mathscr{H}(t - |x|)R_1(t,x)\bigr] \\
        \left(\partial_t^2 - \Delta\right)\bigl[\mathscr{H}(t - |x|)R_2(t,x)\bigr]
    \end{bmatrix}
    -
    \begin{bmatrix}
        \beta_{11}(x) & \beta_{12}(x) \\
        \beta_{21}(x) & \beta_{22}(x)
    \end{bmatrix}
    \begin{bmatrix}
        \mathscr{H}(t - |x|)\, R_1(t,x) \\ 
        \mathscr{H}(t - |x|)\, R_2(t,x)
    \end{bmatrix} \nonumber\\
    & = A + B + C + D.
\end{align}
The term $C$ can be simplified as follows
\begin{align}\label{Term C}
    & \begin{bmatrix}
       \left(\partial_t^2 -\Delta \right)\bigl[ \mathscr{H}(t-|x|)R_1(t,x)\bigr]\\ \left(\partial_t^2 -\Delta \right)\bigl[ \mathscr{H}(t-|x|)R_2(t,x)\bigr]
    \end{bmatrix}\nonumber\\
& \quad\quad= \begin{bmatrix}
    \delta'(t-|x|)(R_1 -R_1) + 2 \dfrac{\delta(t-|x|)}{4\pi|x|} (|x|\partial_t R_1 + R_1 + x\cdot \nabla R_1) \\
    \delta'(t-|x|)(R_2 -R_2) + 2 \dfrac{\delta(t-|x|)}{4\pi|x|} (|x|\partial_t R_2 + R_2 + x\cdot \nabla R_2) 
\end{bmatrix} \nonumber\\
& \quad\quad\quad + \begin{bmatrix}
      \mathscr{H}(t-|x|) \left(\partial_t^2 -\Delta \right) R_1(t,x)\\ \mathscr{H}(t-|x|)\left(\partial_t^2 -\Delta \right) R_2(t,x)
    \end{bmatrix}.
\end{align}
Using Equations \eqref{Green's function} and \eqref{Term C} in  Equation \eqref{eq:A+B+C+D}, we have
\begin{align*}
    \left(\Box I_{2\times2} -  \mathfrak{P}(x)\right)\overrightarrow{U}(t,x) & = \begin{bmatrix}
        \delta(t,x)\\
        \delta(t,x)
    \end{bmatrix} + 2 \begin{bmatrix}
      \dfrac{\delta(t-|x|)}{4\pi|x|} (|x|\partial_t R_1 + R_1 + x\cdot \nabla R_1) \\
      \dfrac{\delta(t-|x|)}{4\pi|x|} (|x|\partial_t R_2 + R_2 + x\cdot \nabla R_2) 
\end{bmatrix} \\
&  \quad +  \begin{bmatrix}
      \mathscr{H}(t-|x|) \left(\partial_t^2 -\Delta \right) R_1(t,x)\\ \mathscr{H}(t-|x|)\left(\partial_t^2 -\Delta \right) R_2(t,x)
    \end{bmatrix} - 
    \begin{bmatrix}
        \left(\beta_{11}(x) + \beta_{12}(x)\right)\dfrac{\delta(t - |x|)}{4\pi|x|} \\
        \left(\beta_{21}(x) + \beta_{22}(x)\right)\dfrac{\delta(t - |x|)}{4\pi|x|}
    \end{bmatrix} \\
    & \quad - \begin{bmatrix}
       \left(\beta_{11}(x) + \beta_{12}(x)\right) \mathscr{H}(t-|x|)R_1(t,x)\\ \left(\beta_{21}(x) + \beta_{22}(x)\right)\mathscr{H}(t-|x|) R_2(t,x)
    \end{bmatrix}.
\end{align*}
Consequently, $\overrightarrow{U}$ solves the point-source problem
\eqref{system of IVP} provided $\overrightarrow{R}$ solves the following Goursat BVP 
\begin{align*}
\begin{cases}
     \left(\Box I_{2\times2} -  \mathfrak{P}(x)\right)\overrightarrow{R}(t,x) = \overrightarrow{0},\quad  x\in \mathbb{R}^3, t> |x|,  \\
    \left(|x|\partial_t  + 1 + x\cdot \nabla \right)\overrightarrow{R}(t,x)= \frac{1}{8\pi}\mathfrak{P}\overrightarrow{\mathscr{A}}, \quad  x\in \mathbb{R}^3, t = |x|.
\end{cases}
\end{align*}
  Now if  we define $\overrightarrow{G}(x):=|x|\overrightarrow{R}(|x|,x)$, then using the chain rule yields that 
    \begin{align*}
     \begin{bmatrix}
         \frac{x}{|x|}\cdot \nabla G_{1}\\[2pt]
         \frac{x}{|x|}\cdot \nabla G_{2}
     \end{bmatrix}    = \begin{bmatrix}
         \left(|x|\partial_t  + 1 + x\cdot \nabla \right){R}_1\\[2pt]
         \left(|x|\partial_t  + 1 + x\cdot \nabla \right){R}_2
     \end{bmatrix} = \frac{1}{8\pi}\mathfrak{P}\overrightarrow{\mathscr{A}}.
    \end{align*}
Now set $\omega:=\dfrac{x}{|x|}$ and define the unit-speed straight line from origin (i.e. $ O = (0,0,0)$) to $x$ by
\begin{align*}
    \gamma(s):=s\,\omega,\qquad 0\le s\le |x|.
\end{align*}
Using the chain rule, we have
\begin{align*}
\begin{bmatrix}
          \frac{d}{ds} G_{1}(\beta(s))\\[2pt]
           \frac{d}{ds}G_{2}(\gamma(s))
     \end{bmatrix}
=\begin{bmatrix}
          \nabla G_{1}(\gamma(s))\cdot\gamma'(s)\\[2pt]
          \nabla G_{2}(\gamma(s))\cdot\gamma'(s)
     \end{bmatrix}
=\begin{bmatrix}
         \omega\cdot \nabla G_{1}\\[2pt]
         \omega\cdot \nabla G_{2}
     \end{bmatrix} = \frac{1}{8\pi}\mathfrak{P}\bigl(s \omega\bigr)\overrightarrow{\mathscr{A}}.
\end{align*}
Integrating from $0$ to $|x|$ yields
\begin{align*}
\overrightarrow{G}(x)-\overrightarrow{G}(O)
&=\int_{0}^{|x|}\frac{d}{ds}\overrightarrow{G}(\gamma(s))\,ds
=\frac{1}{8\pi}\int_{0}^{|x|} \frac{1}{8\pi}\mathfrak{P}\bigl(s a\bigr)\overrightarrow{\mathscr{A}}\,ds .
\end{align*}
Since $\overrightarrow{G}(O)=0$, we get
\begin{align*}
    \overrightarrow{G}(x)=\frac{1}{8\pi}\int_{0}^{|x|}
\mathfrak{P}\!\left(s\,\frac{x}{|x|}\right)\overrightarrow{\mathscr{A}}\,ds.
\end{align*}
Recalling $\overrightarrow{G}(x)=|x|\, \overrightarrow{R}\bigl(|x|,x\bigr)$ and dividing by $|x|$, we have
\begin{align*}
    \overrightarrow{R}\bigl(x,|x|\bigr)
=\frac{1}{8\pi|x|}\int_{0}^{|x|}
\mathfrak{P}\!\left(s\,\frac{x}{|x|}\right)\overrightarrow{\mathscr{A}}\,ds .
\end{align*}
With the substitution $s=|x|t$ (so $t\in[0,1]$),  it follows that
\begin{align*}
    \overrightarrow{R}\bigl(|x|,x\bigr)
&=\frac{1}{8\pi}\int_{0}^{1} \mathfrak{P}\bigl(sx\bigr)\overrightarrow{\mathscr{A}}\,ds,\nonumber\\
&=  \begin{bmatrix}
      \frac{1}{8\pi} \int_0^1 (\beta_{11} + \beta_{12})(tx)\,dt\\[2pt]
      \frac{1}{8\pi} \int_0^1 (\beta_{21} + \beta_{22})(tx)\,dt
    \end{bmatrix}.
\end{align*}
Thus, in order to prove the solutions to \eqref{system of IVP} have the form given by Equation \eqref{Ex for u vector}, we must show the existence of a solution to the following Goursat BVP 
 \begin{align}\label{P;Goursat problem}
     \begin{cases}
         \left(\Box I_{2\times2} -  \mathfrak{P}(x)\right)\overrightarrow{R}(t,x) = \overrightarrow{0},\quad&  x\in \mathbb{R}^3, t> |x|,  \\
    \overrightarrow{R}(t,x)  = \begin{bmatrix}
      \frac{1}{8\pi}  \int_0^1 (\beta_{11} + \beta_{12})(sx)ds\\[2pt]
      \frac{1}{8\pi}  \int_0^1 (\beta_{21} + \beta_{22})(sx)ds
    \end{bmatrix} \quad  &x\in \mathbb{R}^3, t = |x|. 
     \end{cases}
 \end{align}
 Now, since the coupling is due to a  zeroth order term only, therefore following the arguments used in \cite{blaasten2017well,Friedlander_1975}, it can be shown that \eqref{P;Goursat problem} 
 possesses a unique $C^{1}$ solution for a 
sufficiently smooth matrix potential $\mathfrak{P}$. This established the existence for a unique Fundamental solution $\overrightarrow{U}(t,x)$ of the form given by Equation \eqref{Ex for u vector},  to the operator $\Box I_{2\times2}-\mathfrak{P}$ such that $\overrightarrow{U}(t,x)=0$, for  $t<0$. This concludes the proof of the following theorem. 
\begin{theorem}\label{THM1 FS origin}
    Consider $\mathfrak{P}\in C^{\infty}(\mathbb{R}^3)$ is a matrix-valued real potential and $\overrightarrow{U}$ be a solution to the IVP given by Equation \eqref{system of IVP}. Then $\overrightarrow{U}$ is given by \[\overrightarrow{U}(t,x) = \frac{\delta(t-|x|)}{4\pi|x|}\overrightarrow{\mathscr{A}} + \overrightarrow{R}(t,x),\] where $\overrightarrow{R}(t,x)=0$ in the region $t<\lvert x\rvert$, and for the region $t>\lvert x\rvert$, $\overrightarrow{R}$ solves the following Goursat BVP 
    \begin{align*}
        \begin{aligned}
\begin{cases}
         \left(\Box I_{2\times2} -  \mathfrak{P}(x)\right)\overrightarrow{R}(t,x) = \overrightarrow{0},\quad&  x\in \mathbb{R}^3, t> |x|,  \\
    \overrightarrow{R}(t,x)  =  \begin{bmatrix}
      \frac{1}{8\pi} \int_0^1 (\beta_{11} + \beta_{12})(sx)ds\\
       \frac{1}{8\pi}\int_0^1 (\beta_{21} + \beta_{22})(sx)ds
    \end{bmatrix} \quad  &x\in \mathbb{R}^3, t = |x|. 
     \end{cases}
        \end{aligned}
    \end{align*}\end{theorem}

 
\subsection{The point-source at {$e$}}
In this subsection, we provide an expression of the fundamental solution for the coupled system of wave equations with a point-source located at point $e=(1,0,0)$.
The solution to this problem is obtained by translating 
the solution given by Equation ~\eqref{Ex for u vector} to the IVP \eqref{system of IVP} by the vector $e$. More precisely, the following theorem can be established using the above-mentioned translation.  
\begin{theorem}\label{THM1 FS non origin}
\label{thm:;uniqueness}
Let $\mathfrak{P}$ be $2\times2$ matrix with smooth entries and $e$ is a unit vector in $\mathbb{R}^3$. Then the point-source problem 
\begin{align}\label{system of IBVP; source}
    \begin{cases}
        (\Box I_{2\times2} -  \mathfrak{P}(x))\overrightarrow{W}(t,x)  = {\delta}(t,x-e)\overrightarrow{\mathscr{A}},&\quad(t,x)\in \mathbb{R} \times\mathbb{R}^3, \\
       \overrightarrow{W}(t,x)  = \overrightarrow{0}, &\quad (t,x)\in (-\infty, 0)\times\mathbb{R}^3.
    \end{cases}
\end{align}
has a unique solution $\overrightarrow {w} $ having the following expression 
\begin{align*}
     \overrightarrow{W}(t,x) = \frac{\delta(t-|x-e|)}{4\pi|x-e|}\overrightarrow{\mathscr{A}} + \overrightarrow{R}(t,x),
\end{align*}
where 
   $\overrightarrow{R} = 
    \begin{bmatrix}
        R_1\\R_2
    \end{bmatrix},$
  with $R_i\in C^{1}(\mathbb{R}\times \mathbb{R}^3)$ vanishes in the region $t<\lvert x-e\rvert$, and for the region $t>\lvert x-e\rvert$, $\overrightarrow{R}$  solves the following Goursat BVP 
     \begin{align*}
         \begin{cases}
     \left(\Box I_{2\times2} -  \mathfrak{P}(x)\right)\overrightarrow{R}(t,x) = \overrightarrow{0},\quad  x\in \mathbb{R}^3,\  t> |x-e|,  \\
    \overrightarrow{R}(t,x) = \begin{bmatrix}
        \frac{1}{8\pi}\int_{0}^{1}(\beta_{11}+\beta_{12})(sx+(1-s)e)ds\\[2pt]
        \frac{1}{8\pi}\int_{0}^{1}(\beta_{21}+\beta_{22})(sx+(1-s)e)ds
    \end{bmatrix}, \quad  x\in \mathbb{R}^3, \ t = |x-e|. 
\end{cases}
     \end{align*}
\end{theorem}
\begin{proof}
   We first apply the operator $\Box I_{2\times2} -  \mathfrak{P}(x)$ to the ansatz
\begin{align}
    \overrightarrow{W}(t,x) = \frac{\delta(t-|x-e|)}{4\pi|x-e|}\overrightarrow{\mathscr{A}} + \mathscr{H}(t-|x-e|)\overrightarrow{R}(t,x),
\end{align}
where
\begin{align*}
    \overrightarrow{\mathscr{A}} = 
    \begin{bmatrix}
        1\\1
    \end{bmatrix} \quad \text{and} \quad\overrightarrow{R} = 
    \begin{bmatrix}
        R_1\\R_2
    \end{bmatrix},
\end{align*}
gives
\begin{align}\label{eq:I+J+K+L}
    &\left(\Box I_{2\times2} - \mathfrak{P}(x)\right)\overrightarrow{W}(t,x) = 
    \begin{bmatrix}
        \partial_t^2 - \Delta - \beta_{11}(x) & -\beta_{12}(x) \\
        -\beta_{21}(x) & \partial_t^2 - \Delta - \beta_{22}(x)
    \end{bmatrix}
    \begin{bmatrix}
        W_1(t,x) \\ W_2(t,x)
    \end{bmatrix}, \nonumber\\
    &= 
    \begin{bmatrix}
        \partial_t^2 - \Delta - \beta_{11}(x) & -\beta_{12}(x) \\
        -\beta_{21}(x) & \partial_t^2 - \Delta - \beta_{22}(x)
    \end{bmatrix}
    \left\{
        \frac{\delta(t - |x - e|)}{4\pi|x - e|}
        \begin{bmatrix}
            1 \\ 1
        \end{bmatrix}
        + \mathscr{H}(t - |x - e|)
        \begin{bmatrix}
            R_1(t,x) \\ R_2(t,x)
        \end{bmatrix}
    \right\}, \nonumber\\
    &=
    \begin{bmatrix}
        \left(\partial_t^2 - \Delta\right)\dfrac{\delta(t - |x - e|)}{4\pi|x - e|} \\
        \left(\partial_t^2 - \Delta\right)\dfrac{\delta(t - |x - e|)}{4\pi|x - e|}
    \end{bmatrix}
    - 
    \begin{bmatrix}
        \left(\beta_{11}(x) + \beta_{12}(x)\right)\dfrac{\delta(t - |x - e|)}{4\pi|x - e|} \\
        \left(\beta_{21}(x) + \beta_{22}(x)\right)\dfrac{\delta(t - |x - e|)}{4\pi|x - e|}
    \end{bmatrix} \nonumber\\
    &\quad +
    \begin{bmatrix}
        \left(\partial_t^2 - \Delta\right)\bigl[\mathscr{H}(t - |x - e|)R_1(t,x)\bigr] \\
        \left(\partial_t^2 - \Delta\right)\bigl[\mathscr{H}(t - |x - e|)R_2(t,x)\bigr]
    \end{bmatrix}
    -
    \begin{bmatrix}
        \beta_{11}(x) & \beta_{12}(x) \\
        \beta_{21}(x) & \beta_{22}(x)
    \end{bmatrix}
    \begin{bmatrix}
        \mathscr{H}(t - |x - e|)\, R_1(t,x) \\ 
        \mathscr{H}(t - |x - e|)\, R_2(t,x)
    \end{bmatrix}, \nonumber\\
    & = I + J + K + L.
\end{align}
Now, we simplify the term $K$
\begin{align}\label{Term K}
    & \begin{bmatrix}
       \left(\partial_t^2 -\Delta \right) \mathscr{H}(t-|x-e|)R_1(t,x)\\ \left(\partial_t^2 -\Delta \right)\mathscr{H}(t-|x-e|) R_2(t,x)
    \end{bmatrix}\nonumber\\
& \quad\quad= \begin{bmatrix}
    \delta'(t-|x-e|)(R_1 -R_1) + 2 \frac{\delta(t-|x-e|)}{4\pi|x-e|} (|x-e|\partial_t R_1 + R_1 + (x-e)\cdot \nabla R_1) \\
    \delta'(t-|x-e|)(R_2 -R_2) + 2 \frac{\delta(t-|x-e|)}{4\pi|x-e|} (|x-e|\partial_t R_2 + R_2 + (x-e)\cdot \nabla R_2) 
\end{bmatrix} \nonumber\\
& \quad\quad\quad + \begin{bmatrix}
      \mathscr{H}(t-|x-e|) \left(\partial_t^2 -\Delta \right) R_1(t,x)\\ \mathscr{H}(t-|x-e|)\left(\partial_t^2 -\Delta \right) R_2(t,x)
    \end{bmatrix}.
\end{align}
Observe that the first term corresponds to the Green's function associated with the operator $\partial_t^2 - \Delta$, see \cite{Friedlander_1975}, therefore, we have 
\begin{align}\label{Green function}
    (\partial_t^2 - \Delta)\frac{\delta(t-|x-e|)}{4\pi|x-e|} = \delta(t,x-e).
\end{align}
Using equations \eqref{Green function}, \eqref{Term K} and \eqref{eq:I+J+K+L}, we get
\begin{align*}
    \left(\Box I_{2\times2} -  \mathfrak{P}(x)\right)&\overrightarrow{W}(t,x)  = \begin{bmatrix}
        \delta(t,x-e)\\
        \delta(t,x-e)
    \end{bmatrix} + 2 \begin{bmatrix}
      \frac{\delta(t-|x-e|)}{4\pi|x-e|} (|x-e|\partial_t R_1 + R_1 + (x-e)\cdot \nabla R_1) \\
      \frac{\delta(t-|x-e|)}{4\pi|x-e|} (|x-e|\partial_t R_2 + R_2 + (x-e)\cdot \nabla R_2) 
\end{bmatrix} \\
& +  \begin{bmatrix}
      \mathscr{H}(t-|x-e|) \left(\partial_t^2 -\Delta \right) R_1(t,x)\\ \mathscr{H}(t-|x-e|)\left(\partial_t^2 -\Delta \right) R_2(t,x)
    \end{bmatrix} - 
    \begin{bmatrix}
        \left(\beta_{11}(x) + \beta_{12}(x)\right)\dfrac{\delta(t - |x - e|)}{4\pi|x - e|} \\
        \left(\beta_{21}(x) + \beta_{22}(x)\right)\dfrac{\delta(t - |x - e|)}{4\pi|x - e|}
    \end{bmatrix} \\
    & - \begin{bmatrix}
       \left(\beta_{11}(x) + \beta_{12}(x)\right) \mathscr{H}(t-|x-e|)R_1(t,x)\\ \left(\beta_{21}(x) + \beta_{22}(x)\right)\mathscr{H}(t-|x-e|) R_2(t,x)
    \end{bmatrix}.
\end{align*}
Consequently, $\overrightarrow{W}$ solves the point-source problem
\eqref{system of IVP} provided $\overrightarrow{R}$ solves the following Goursat BVP
\begin{align*}
\begin{cases}
     \left(\Box I_{2\times2} -  \mathfrak{P}(x)\right)\overrightarrow{R}(t,x) = \overrightarrow{0},\quad  x\in \mathbb{R}^3, t> |x-e|,  \\
    \left(|x-e|\partial_t  + 1 + (x-e)\cdot \nabla \right)\overrightarrow{R} = \frac{1}{8\pi}\mathfrak{P}\overrightarrow{\mathscr{A}}, \quad  x\in \mathbb{R}^3, t = |x-e|.
\end{cases}
\end{align*}
   Now if we define $\overrightarrow{\mathcal{G}}(x) := |x-e|\overrightarrow{R}(|x-e|,x)$, then using the chain rule we have 
    \begin{align}
     \begin{bmatrix}
         \frac{x-e}{|x-e|}\cdot \nabla \mathcal{G}_{1}\\
         \frac{x-e}{|x-e|}\cdot \nabla \mathcal{G}_{2}
     \end{bmatrix}    = \begin{bmatrix}
         \left(|x-e|\partial_t  + 1 + (x-e)\cdot \nabla \right){R_1}\\
         \left(|x-e|\partial_t  + 1 + (x-e)\cdot \nabla \right){R_2}
     \end{bmatrix} = \frac{1}{8\pi}\mathfrak{P}\overrightarrow{\mathscr{A}}.
    \end{align}
Now set $\widetilde{\omega}:=\dfrac{x-e}{|x-e|}$ and define the unit-speed straight line from $e$ to $x$ by
\begin{align*}
    \widetilde{\gamma}(s):=e+s\,\widetilde{\omega},\qquad 0\le s\le |x-e|.
\end{align*}
By the chain rule, we have
\begin{align*}
\begin{bmatrix}
          \frac{d}{ds} \mathcal{G}_{1}(\widetilde{\gamma}(s))\\
           \frac{d}{ds}\mathcal{G}_{2}(\widetilde{\gamma}(s))
     \end{bmatrix}
=\begin{bmatrix}
          \nabla \mathcal{G}_{1}(\widetilde{\gamma}(s))\cdot\widetilde{\gamma}'(s)\\
          \nabla \mathcal{G}_{2}(\widetilde{\gamma}(s))\cdot\widetilde{\gamma}'(s)
     \end{bmatrix}
=\begin{bmatrix}
         \widetilde{\omega}\cdot \nabla \mathcal{G}_{1}\\
         \widetilde{\omega}\cdot \nabla \mathcal{G}_{2}
     \end{bmatrix} = \frac{1}{8\pi}\mathfrak{P}\bigl(e+s \widetilde{\omega}\bigr)\overrightarrow{\mathscr{A}}.
\end{align*}
Integrating from $0$ to $|x-e|$ yields
\begin{align*}
\overrightarrow{\mathcal{G}}(x)-\overrightarrow{\mathcal{G}}(e)
&=\int_{0}^{|x-e|}\frac{d}{ds}\overrightarrow{\mathcal{G}}(\widetilde{\gamma}(s))\,ds
=\frac{1}{8\pi}\int_{0}^{|x-e|} \frac{1}{8\pi}\mathfrak{P}\bigl(e+s \widetilde{\omega}\bigr)\overrightarrow{\mathscr{A}}\,ds .
\end{align*}
Since $\overrightarrow{\mathcal{G}}(e)=|e-e|\,\overrightarrow{R}(0,e)=0$, we get
\begin{align*}
    \overrightarrow{\mathcal{G}}(x)=\frac{1}{8\pi}\int_{0}^{|x-e|}
\mathfrak{P}\!\left(e+s\,\frac{x-e}{|x-e|}\right)\overrightarrow{\mathscr{A}}\,ds.
\end{align*}
Recalling $\overrightarrow{\mathcal{G}}(x)=|x-e|\, \overrightarrow{R}\bigl(|x-e|,x\bigr)$ and dividing by $|x-e|$, we have
\begin{align}
    \overrightarrow{R}\bigl(|x-e|,x\bigr)
=\frac{1}{8\pi|x-e|}\int_{0}^{|x-e|}
\mathfrak{P}\!\left(e+s\,\frac{x-e}{|x-e|}\right)\overrightarrow{\mathscr{A}}\,ds .
\end{align}
With the substitution $s=|x-e|t$ (so $t\in[0,1]$), we obtain 
\begin{align}
    \overrightarrow{R}\bigl(|x-e|,x\bigr)
&=\frac{1}{8\pi}\int_{0}^{1} \mathfrak{P}\bigl(e+t(x-e)\bigr)\overrightarrow{\mathscr{A}}\,dt,\nonumber\\
&=  \begin{bmatrix}
      \frac{1}{8\pi} \int_0^1 (\beta_{11} + \beta_{12})(e+t(x-e))ds\\[2pt]
      \frac{1}{8\pi} \int_0^1 (\beta_{21} + \beta_{22})(e+t(x-e))ds
    \end{bmatrix}.
\end{align}
Using arguments similar to those in the previous theorem, the existence of $\overrightarrow{R}$ is established. This concludes the proof.
\end{proof}


   \section{Proof of main theorems}\label{sec_proofs_main_results}
 In this section, we present the proofs of the main Theorems \ref{thm1}-\ref{thm4} and it is divided into two subsections; the first contains the comparable case where we give the proof of Theorems \ref{thm1} and \ref{thm3} and the second one consists of the proof of unique determination of coefficients under symmetry conditions stated in Theorems \ref{thm2} and \ref{thm4}. 
 In the proofs, we work with the following form of the solution to the point-source problem, given by IVP \eqref{system of IVP}.
Specifically, we absorb the Heaviside function into the regular part of the solution and write the solution of the IVP \eqref{system of IVP} given in  Equation \eqref{Expression for u vector}   as follows
\begin{align}\label{Solution; cpt form}
\overrightarrow{U}(t,x)
=
\frac{\delta(t-|x|)}{4\pi|x|}\,\overrightarrow{\mathscr{A}}
+
\overrightarrow{R}^{u}(t,x),
\end{align}
where $\overrightarrow{R}^{u}(t,x)=\overrightarrow{0}$ for $t<|x|$.
For $t>|x|$, $\overrightarrow{R}^{u}$ is a $C^{1}$ solution for the
following the Goursat problem
\begin{align*}
\begin{cases}
\bigl(\Box I_{2\times2}-\mathfrak{P}(x)\bigr)\overrightarrow{R}^{u}(t,x)
= \overrightarrow{0},
& x\in\mathbb{R}^{3},\ t>|x|, \\[1mm]
\bigl(|x|\partial_t + 1 + x\cdot\nabla\bigr)\overrightarrow{R}^{u}(t,x)
= \dfrac{1}{8\pi}\mathfrak{P}(x)\overrightarrow{\mathscr{A}},
& x\in\mathbb{R}^{3},\ t=|x|.
\end{cases}
\end{align*}
 This representation is particularly convenient for deriving integral identities and for applying integration by parts in the subsequent analysis, since $\overrightarrow{R}^{u}$ vanishes identically for $t<|x|$ and is smooth in the region $t>|x|$.

We begin by recalling some preliminary known results related to parametrization of the sphere and the ellipsoid, along with their surface measures, which will be used in the proofs of
Theorems~\ref{thm1}, \ref{thm2}, \ref{thm3} and~\ref{thm4}.
A point $x\in\mathbb{R}^3$ is represented in the spherical coordinates $(r,\theta,\phi)$ as
\begin{align*}
x=\bigl(r\sin\phi\cos\theta,\; r\sin\phi\sin\theta,\; r\cos\phi\bigr),
\end{align*}
where $r\ge0$, $0\le\theta<2\pi$, and $0\le\phi\le\pi$.
The corresponding volume element is given by
\begin{align}\label{volume element_sphere}
dx = r^{2}\sin\phi\,dr\,d\theta\,d\phi,
\end{align}
while the surface element on the sphere $\{x\in\mathbb{R}^3:|x|=r\}$ is
\begin{align}\label{surface element_sphere}
dS_x = r^{2}\sin\phi\,d\theta\,d\phi.
\end{align}
Next, we present a lemma that yields the parametrization in prolate spheroidal coordinates associated with an ellipsoid having the foci at $0$ and $e$.
\begin{lemma}\cite[Lemma~3.3]{w1} \label{lem3.1}
    Let $e=(1,0,0)$ and $x=(x_1,x_2,x_3)\in \mathbb{R}^3$ and consider the solid ellipsoid region
$|x-e|+|x|\leq r$. This solid ellipsoid admits a parametrization in prolate spheroidal coordinates $(\rho,\theta,\phi)$ given by
\begin{align}\label{e13}
\begin{cases}
x_1 = \dfrac{1}{2} + \dfrac{1}{2}\cosh\rho \cos\phi,\\[6pt]
x_2 = \dfrac{1}{2}\sinh\rho \sin\theta \sin\phi,\\[6pt]
x_3 = \dfrac{1}{2}\sinh\rho \cos\theta \sin\phi,
\end{cases}
\end{align}
where $\cosh\rho \leq r$, $\theta \in (0,2\pi)$, and $\phi \in (0,\pi)$. The surface measure denoted by $\,dS_x$ on the boundary $|x-e|+|x|=r$ is
\begin{align}\label{e14}
dS_x = \frac{1}{4}\sinh\rho \sin\phi \sqrt{\cosh^2\rho - \cos^2\phi}\, d\theta\, d\phi,
\end{align}
with $\cosh\rho = r$, $\theta \in [0,2\pi]$, and $\phi \in [0,\pi]$.
The corresponding volume element denoted by $\,dx$ in the region $|x-e|+|x|\leq r$ is
\begin{align}\label{e15}
dx = \frac{1}{8}\sinh\rho \sin\phi \left(\cosh^2\rho - \cos^2\phi\right)\, d\rho\, d\theta\, d\phi,
\end{align}
where $\cosh\rho \leq r$, $\theta \in [0,2\pi]$, and $\phi \in [0,\pi]$.
\end{lemma}
We now proceed with the proofs of the main Theorems of this article. We will do this in the next two subsections. 
\subsection{Proof of Theorems \ref{thm1} and \ref{thm3}}\label{Sec; proofs comparable results}
 In this subsection, we present the proofs of Theorems \ref{thm1} and \ref{thm3}, both of which are stated under the comparability assumption on the matrix-valued coefficients. 
\subsubsection{Proof of Theorem~\ref{thm1}}
Denote  $\mathfrak{P}(x):= \mathfrak{P}^{(1)}(x) - \mathfrak{P}^{(2)}(x) $ and $ \overrightarrow{U}(t,x) := \overrightarrow{U}^{(1)}(t,x) - \overrightarrow{U}^{(2)}(t,x)$
where $\mathfrak{P}^{(k)}$ and $\overrightarrow{U}^{(k)}$ for $k=1,2$, are same as in statement of Theorem \ref{thm1}. Now under the hypothesis of Theorem \ref{thm1}, we have that  $\overrightarrow{U}(t,0)=0$, for each $0\leq t\leq T$ and $\overrightarrow{U}$ also solves the following nonhomogeneous IVP
\begin{align}\label{IBVP u}
\begin{cases}
\bigl(\Box I_{2\times2}-\mathfrak{P}^{(1)}(x)\bigr)\overrightarrow{U}(t,x)
= \mathfrak{P}(x)\overrightarrow{U}^{(2)}(t,x), & (t,x)\in\mathbb{R}\times\mathbb{R}^3,\\[2mm]
\overrightarrow{U}(t,x)=\overrightarrow{0},
& (t,x)\in(-\infty, 0)\times\mathbb{R}^3.
\end{cases}
\end{align}
Let $\overrightarrow{W}$ denote the solution of the following adjoint problem
\begin{align}\label{e8}
\begin{cases}
\bigl(\Box I_{2\times2}-[\mathfrak{P}^{(1)}(x)]^{t}\bigr)\overrightarrow{W}(t,x)
= \delta(t,x)\overrightarrow{\mathscr{A}}\,, & (t,x)\in\mathbb{R}\times\mathbb{R}^3,\\[2mm]
\overrightarrow{W}(t,x)=\overrightarrow{0},
& (t,x)\in(-\infty, 0)\times\mathbb{R}^3.
\end{cases}
\end{align}  
Multiplying the governing equation for $\overrightarrow{U}$ in the IVP \eqref{IBVP u} by $\overrightarrow{W}(2\tau-t,x)$, $\tau\in \mathbb{R}$, and integrating over $\mathbb{R}\times\mathbb{R}^3$, we obtain
\begin{align}\label{e7}
\int_{\mathbb{R}}\int_{\mathbb{R}^3}
\bigl(\Box I_{2\times2}-\mathfrak{P}^{(1)}(x)\bigr)\overrightarrow{U}(t,x)\cdot
\overrightarrow{W}(2\tau-t,x)\,dx\,dt
=
\int_{\mathbb{R}}\int_{\mathbb{R}^3}
\left[\mathfrak{P}(x)\overrightarrow{U}^{(2)}(t,x)\right]\cdot
\overrightarrow{W}(2\tau-t,x)\,dx\,dt .
\end{align}
Using integration by parts, together with the vanishing initial conditions, and the fact that
$\overrightarrow{U}(t,x)=
\overrightarrow{W}(t,x)=\overrightarrow{0}$ for $t<|x|$, we obtain the following simplified form of  \eqref{e7}
\begin{align*}
\int_{\mathbb{R}}\int_{\mathbb{R}^3}
\overrightarrow{U}(t,x)\cdot
\bigl(\Box I_{2\times2}-[\mathfrak{P}^{(1)}(x)]^{t}\bigr)
\overrightarrow{W}(2\tau-t,x)\,dx\,dt
&=
\int_{\mathbb{R}}\int_{\mathbb{R}^3}
\overrightarrow{U}(t,x)\cdot
\overrightarrow{\mathscr{A}}\,\delta(2\tau-t,x)\,dx\,dt\\
&= U_1(2\tau,0)+U_2(2\tau,0).
\end{align*}
Now using the fact that   $\overrightarrow{U}(2\tau,0)=\overrightarrow{0}$, for each $\tau\in\left[0,\dfrac{T}{2}\right]$, therefore  we conclude that    
 the left-hand side of the Equation
\eqref{e7} is equal to zero.  Thus, we get  
\begin{align}\label{e9}
\int_{\mathbb{R}}\int_{\mathbb{R}^3}
\bigl(\Box I_{2\times2}-\mathfrak{P}^{(1)}(x)\bigr)\overrightarrow{U}(t,x)\cdot
\overrightarrow{W}(2\tau-t,x)\,dx\,dt
=0,
\quad \mbox{for any}\ \tau\in\left[0,\frac{T}{2}\right].
\end{align}
Using this in Equation \eqref{e7}, we obtain
\begin{align*}
0
&=\int_{\mathbb{R}}\int_{\mathbb{R}^3}
\bigl[\mathfrak{P}(x)\overrightarrow{U}^{(2)}(t,x)\bigr]\cdot
\overrightarrow{W}(2\tau-t,x)\,dx\,dt, \ \mbox{for any}\ \tau\in\left[0,\frac{T}{2}\right].
\end{align*}
Substitute the explicit representations of $\overrightarrow{U}^{(2)}$ and
$\overrightarrow{W}$ similar to that of  established in Theorem \ref{THM1 FS origin}, to obtain 
\begin{align*}
0
&=\int_{\mathbb{R}}\int_{\mathbb{R}^3}
\left(\begin{bmatrix}
\beta_{11}(x) & \beta_{12}(x)\\
\beta_{21}(x) & \beta_{22}(x)
\end{bmatrix}
\begin{bmatrix}
\dfrac{\delta(t-|x|)}{4\pi|x|}\\[2mm]
\dfrac{\delta(t-|x|)}{4\pi|x|}
\end{bmatrix}\right)
\cdot
\begin{bmatrix}
\dfrac{\delta(2\tau-t-|x|)}{4\pi|x|}\\[2mm]
\dfrac{\delta(2\tau-t-|x|)}{4\pi|x|}
\end{bmatrix}
\,dx\,dt \\
&\quad
+\int_{\mathbb{R}}\int_{\mathbb{R}^3}
\left(\begin{bmatrix}
\beta_{11}(x) & \beta_{12}(x)\\
\beta_{21}(x) & \beta_{22}(x)
\end{bmatrix}
\begin{bmatrix}
\dfrac{\delta(t-|x|)}{4\pi|x|}\\[2mm]
\dfrac{\delta(t-|x|)}{4\pi|x|}
\end{bmatrix}\right)
\cdot
\begin{bmatrix}
R^{w}_1(2\tau-t,x)\\
R^{w}_2(2\tau-t,x)
\end{bmatrix}
\,dx\,dt \\
&\quad
+\int_{\mathbb{R}}\int_{\mathbb{R}^3}
\left(\begin{bmatrix}
\beta_{11}(x) & \beta_{12}(x)\\
\beta_{21}(x) & \beta_{22}(x)
\end{bmatrix}
\begin{bmatrix}
R_1^{u^{(2)}}(t,x)\\
R_2^{u^{(2)}}(t,x)
\end{bmatrix}
\right)
\cdot
\begin{bmatrix}
\dfrac{\delta(2\tau-t-|x|)}{4\pi|x|}\\[2mm]
\dfrac{\delta(2\tau-t-|x|)}{4\pi|x|}
\end{bmatrix}
\,dx\,dt \\
&\quad
+\int_{\mathbb{R}}\int_{\mathbb{R}^3}
\left(\begin{bmatrix}
\beta_{11}(x) & \beta_{12}(x)\\
\beta_{21}(x) & \beta_{22}(x)
\end{bmatrix}
\begin{bmatrix}
R_1^{u^{(2)}}(t,x)\\
R_2^{u^{(2)}}(t,x)
\end{bmatrix}
\right)
\cdot
\begin{bmatrix}
R^{w}_1(2\tau-t,x)\\
R^{w}_2(2\tau-t,x)
\end{bmatrix}
\,dx\,dt .
\end{align*}
Using the fact that $\overrightarrow{R}^{{u}^{(2)}} = \overrightarrow{0}$ and  $\overrightarrow{R}^{w} = \overrightarrow{0}$ 
 for $t<|x|$, together with the following standard identity
\begin{align*}
  \int_{\mathbb{R}^n}\Psi(x)\,\delta(L(x))\,dx
=\int_{L(x)=0}\frac{\Psi(x)}{|\nabla_x L(x)|}\,dS_x,  
\end{align*}
where $dS_x$ denotes the surface measure on the level set $L(x)=0$,  the above expression reduces to
\begin{align*}
0
&=\frac{1}{16\pi^2}
\int_{|x|=\tau}
\frac{1}{|x|^2|\nabla_x(2\tau-2|x|)|}
\left(\begin{bmatrix}
\beta_{11}(x) & \beta_{12}(x)\\
\beta_{21}(x) & \beta_{22}(x)
\end{bmatrix}
\begin{bmatrix}
1\\
1
\end{bmatrix}\right)
\cdot
\begin{bmatrix}
1\\
1
\end{bmatrix}
\,dS_x \\
&\quad
+\frac{1}{4\pi}\int_{|x|\le\tau}
\frac{1}{|x|}
\left(\begin{bmatrix}
\beta_{11}(x) & \beta_{12}(x)\\
\beta_{21}(x) & \beta_{22}(x)
\end{bmatrix}
\begin{bmatrix}
1\\
1
\end{bmatrix}\right)
\cdot
\begin{bmatrix}
R^{w}_1(2\tau-|x|,x)\\
R^{w}_2(2\tau-|x|,x)
\end{bmatrix}
\,dx \\
&\quad
+\frac{1}{4\pi}\int_{|x|\le\tau}
\frac{1}{|x|}
\left(\begin{bmatrix}
\beta_{11}(x) & \beta_{12}(x)\\
\beta_{21}(x) & \beta_{22}(x)
\end{bmatrix}
\begin{bmatrix}
R^{u^{(2)}}_1(2\tau-|x|,x)\\
R^{u^{(2)}}_2(2\tau-|x|,x)
\end{bmatrix}\right)
\cdot
\begin{bmatrix}
1\\
1
\end{bmatrix}
\,dx \\
&\quad
+\int_{|x|\le\tau}\int_{|x|}^{2\tau-|x|}
\left(\begin{bmatrix}
\beta_{11}(x) & \beta_{12}(x)\\
\beta_{21}(x) & \beta_{22}(x)
\end{bmatrix}
\begin{bmatrix}
R^{u^{(2)}}_1(t,x)\\
R^{u^{(2)}}_2(t,x)
\end{bmatrix}\right)
\cdot
\begin{bmatrix}
R^{w}_1(2\tau-t,x)\\
R^{w}_2(2\tau-t,x)
\end{bmatrix}
\,dt\,dx.
\end{align*}
Since $\tau\in\left[0,\dfrac{T}{2}\right]$, and both $\overrightarrow{R}^{{u}^{(2)}}$ and
$\overrightarrow{R^{w}}$ are bounded on a compact subset, the above identity yields the estimate
\begin{align}\label{e10}
\int_{|x|=\tau}\frac{1}{|x|^2}
\left( \sum\limits_{i,j = 1}^{2}\beta_{ij}(x) \right)\,dS_x
\le
K\int_{|x|\le\tau}\frac{1}{|x|}
\left( \sum\limits_{i,j = 1}^{2}\beta_{ij}(x) \right)\,dx ,\ \mbox{for}\ \tau\in \left[0,\frac{T}{2}\right]
\end{align}
and  some constant $K>0$. In the above inequality, we have used that $|\nabla_x(\tau-|x|)| = 1$. Now, define\begin{align}\label{e11}
I(\tau):=\int_{|x|=\tau}\frac{1}{|x|^2}
\left( \sum\limits_{i,j = 1}^{2}\beta_{ij}(x) \right)\,dS_x .
\end{align} 
and using this in Equation \eqref{e10} along with polar coordinates,  to arrive at 
\begin{align*}
I(\tau)
&\le
K\int_{|x|\le\tau}\frac{1}{|x|}
\left( \sum\limits_{i,j = 1}^{2}\beta_{ij}(x) \right)\,dx
\le
KC\int_{0}^{\tau} I(r)\,dr, \ \mbox{for any}\ \tau\in\left[0,\frac{T}{2}\right].
\end{align*}
Now since $\beta_{ij}\geq 0$, for each $1\leq i,j\leq 2$, therefore we have $I(\tau)\geq 0$, for any $0\leq \tau\leq \frac{T}{2}$, hence after applying the  Gr\"onwall’s inequality, we conclude 
\begin{align*}
 I(\tau)=0, \quad  \mbox{for any}\ \tau\in\left[0,\dfrac{T}{2}\right].   
\end{align*}
Finally, since each entry $\beta_{ij}(x)$ is non-negative, Equation \eqref{e11} implies
\begin{align*}
   \beta_{ij}(x)=0,\quad \mbox{for any}\  x\in \mathbb{R}^3 \ \mbox{such that}\  |x|\le \frac{T}{2}, \quad 1\le i,j\le2. 
\end{align*}
Therefore, we have
\begin{align*}
   \mathfrak{P}^{(1)}(x)=\mathfrak{P}^{(2)}(x), \quad \mbox{for any}\  x\in \mathbb{R}^3 \ \mbox{such that} \ |x|\le \frac{T}{2}. 
\end{align*}
This concludes the proof.\qed

\subsubsection{Proof of Theorem \ref{thm3}}
Again denote  $\mathfrak{P}(x):= \mathfrak{P}^{(1)}(x) - \mathfrak{P}^{(2)}(x) $ and $ \overrightarrow{U}(t,x) := \overrightarrow{U}^{(1)}(t,x) - \overrightarrow{U}^{(2)}(t,x)$
where $\mathfrak{P}^{(k)}$ and $\overrightarrow{U}^{(k)}$ for $k=1,2$, are same as in statement of Theorem \ref{thm3}. Now under the hypothesis of Theorem \ref{thm3}, we have that  $\overrightarrow{U}(t,e)=0$, for each $0\leq t\leq T$ and $\overrightarrow{U}$ also solves the following nonhomogeneous IVP
\begin{align}\label{IVP-u}
\begin{cases}
    \left(\Box I_{2\times2} -  \mathfrak{P}^{(1)}(x)\right)\overrightarrow{U}(t,x)  =
 \beta(x)\overrightarrow{U}^{(2)}(t,x),\quad&(t,x)\in \mathbb{R} \times\mathbb{R}^3, \\
       \overrightarrow{U}(t,x)  =  \overrightarrow{0},  &(t,x)\in (-\infty, 0)\times\mathbb{R}^3.
\end{cases}
\end{align}
Assume that $\overrightarrow{W}$ is a solution of the following IVP for the adjoint equation
\begin{align}\label{e19}
    \begin{cases}
         \left(\Box I_{2\times2} -  [\mathfrak{P}^{(1)}(x)]^{t}\right)\overrightarrow{W}(t,x)  = \overrightarrow{\mathscr{A}}{\delta}(t,x - e),\quad&(t,x)\in \mathbb{R} \times\mathbb{R}^3, \\
       \overrightarrow{W}(t,x)  = \overrightarrow{0},  &(t,x)\in (-\infty, 0)\times\mathbb{R}^3.  
    \end{cases} 
  \end{align} 
Now, multiplying the governing equation for $\overrightarrow{U}$ in the IVP \eqref{IVP-u} by $\overrightarrow{W}(2\tau-t,x)$, $\tau\in \mathbb{R}$, and integrating over $\mathbb{R}\times\mathbb{R}^3$, we obtain 
\begin{align}\label{e18}
\int\limits_{\mathbb{R}}\int\limits_{\mathbb{R}^3}  \left(\left(\Box I_{2\times2} -  \mathfrak{P}^{(1)}(x)\right)\overrightarrow{U}(t,x)\right)\cdot &\overrightarrow{W}(2\tau-t,x) \,dx \,dt \nonumber\\ &= \int\limits_{\mathbb{R}}\int\limits_{\mathbb{R}^3} \left(\mathfrak{P}(x)\overrightarrow{U}^{(2)}(t,x)\right)\cdot \overrightarrow{W}(2\tau-t,x)\,dx \,dt, 
\end{align} 
By applying the integration by parts, along with the vanishing initial conditions, and the fact that 
$\overrightarrow{U}(t,x)=\overrightarrow{0}$ for $t<|x|$ and
$\overrightarrow{W}(t,x)=\overrightarrow{0}$ for $t<|x-e|$, we obtain the following simplified form of the left-hand side of Equation \eqref{e18}  
\begin{align*}
\int\limits_{\mathbb{R}}\int\limits_{\mathbb{R}^3}  \left((\Box I_{2\times2} -  \mathfrak{P}^{(1)}(x))\overrightarrow{U}(t,x)\right)\cdot &\overrightarrow{W}(2\tau-t,x) \,dx \,dt \nonumber\\ &= \int\limits_{\mathbb{R}}\int\limits_{\mathbb{R}^3} \overrightarrow{U}(t,x)\cdot\left(\Box I_{2\times2} -  [\mathfrak{P}^{(1)}(x)]^{t}\right)  \overrightarrow{W}(2\tau-t,x) \,dx \,dt, \\
    &= \int\limits_{\mathbb{R}}\int\limits_{\mathbb{R}^3} \overrightarrow{U}(t,x)\cdot \overrightarrow{\mathscr{A}}{\delta}(2\tau-t,x-e) \,dx \,dt,\\
    &=  u(2\tau,e).
\end{align*} 
Hypothesis of Theorem \ref{thm3}, gives us   $\overrightarrow{U}(2\tau,e) = \overrightarrow{0}$, for each $\tau\in \left[0,\dfrac{T}{2}\right]$, hence  we obtain
\begin{align*}
     \int\limits_{\mathbb{R}}\int\limits_{\mathbb{R}^3}  \left(\left(\Box I_{2\times2} -  \mathfrak{P}^{(1)}(x)\right)\overrightarrow{U}(t,x)\right)\cdot\overrightarrow{W}(2\tau-t,x) \,dx \,dt = 0.
\end{align*}
After using the above equation  along with the  expressions of solutions $\overrightarrow{U}^{(2)}$ from Theorem \ref{THM1 FS origin} and $\overrightarrow{W}$ from Theorem \ref{THM1 FS non origin},  we get
\begin{align*}
    0 
    &= \int\limits_{\mathbb{R}}\int\limits_{\mathbb{R}^3}
 \left( \begin{bmatrix}
        \beta_{11}(x) & \beta_{12}(x)\\
        \beta_{21}(x) & \beta_{22}(x)
    \end{bmatrix}   
    \begin{bmatrix}
        \dfrac{\delta(t-|x|)}{4\pi|x|}\\
        \dfrac{\delta(t-|x|)}{4\pi|x|}
    \end{bmatrix} \right)
    \cdot \begin{bmatrix}
         \dfrac{\delta(2\tau-t-|x-e|)}{4\pi|x-e|}\\
        \dfrac{\delta(2\tau- t-|x-e|)}{4\pi|x-e|}
    \end{bmatrix}\,dx \,dt \\
    & \quad  + \int\limits_{\mathbb{R}}\int\limits_{\mathbb{R}^3}
  \left(\begin{bmatrix}
        \beta_{11}(x) & \beta_{12}(x)\\
        \beta_{21}(x) & \beta_{22}(x)
    \end{bmatrix}   
    \begin{bmatrix}
        \dfrac{\delta(t-|x|)}{4\pi|x|}\\
        \dfrac{\delta(t-|x|)}{4\pi|x|}
    \end{bmatrix} \right)
    \cdot \begin{bmatrix}
         R^{w}_1(2\tau - t,x)\\
         R^{w}_2(2\tau -t,x )
    \end{bmatrix} \,dx \,dt\\
    & \quad  + \int\limits_{\mathbb{R}}\int\limits_{\mathbb{R}^3}
 \left( \begin{bmatrix}
        \beta_{11}(x) & \beta_{12}(x)\\
        \beta_{21}(x) & \beta_{22}(x)
    \end{bmatrix}   
    \begin{bmatrix}
      R^{u^{(2)}}_1(2\tau - t,x)\\
         R^{u^{(2)}}_2(2\tau -t,x )
    \end{bmatrix} \right)
    \cdot \begin{bmatrix}
        \dfrac{\delta(2\tau-t-|x-e|)}{4\pi|x-e|}\\
        \dfrac{\delta(2\tau- t-|x-e|)}{4\pi|x-e|}
    \end{bmatrix} \,dx \,dt\\
    & \quad  + \int\limits_{\mathbb{R}}\int\limits_{\mathbb{R}^3}
 \left( \begin{bmatrix}
        \beta_{11}(x) & \beta_{12}(x)\\
        \beta_{21}(x) & \beta_{22}(x)
    \end{bmatrix}   
    \begin{bmatrix}
        R^{u^{(2)}}_1(2\tau - t,x)\\
         R^{u^{(2)}}_2(2\tau -t,x )
    \end{bmatrix} \right)
    \cdot \begin{bmatrix}
         R^{w}_1(2\tau - t,x)\\
         R^{w}_2(2\tau -t,x )
    \end{bmatrix}\,dx \,dt. 
\end{align*}
Now, using the fact that $\overrightarrow{R}^{u^{(2)}} = \overrightarrow{0}$ for $t< |x|$, and  $\overrightarrow{R}^{w} = \overrightarrow{0}$ for $t< |x-e|$, together with the following standard identity
\begin{align*}
  \int_{\mathbb{R}^n}\Psi(x)\,\delta(L(x))\,dx
=\int_{L(x)=0}\frac{\Psi(x)}{|\nabla_x L(x)|}\,dS_x,  
\end{align*}
where $\,dS_x$ stands for the surface measure on the surface $L=0$, we get
\begin{align*}
  \quad \quad \quad  0  &= \frac{1}{16\pi^2}\int\limits_{|x-e|+|x| = 2\tau}
  \frac{1}{|x||x-e||\nabla_x(2\tau -|x|-|x-e|)|}
  \left(\begin{bmatrix}
        \beta_{11}(x) & \beta_{12}(x)\\
        \beta_{21}(x) & \beta_{22}(x)
    \end{bmatrix}   
    \begin{bmatrix}
        1 \\
        1
    \end{bmatrix}\right) 
    \cdot \begin{bmatrix}
        1 \\
        1
    \end{bmatrix}\,dS_x  \\
    & \quad  + \frac{1}{4\pi} \int\limits_{|x-e|+|x| \leq 2\tau} \frac{1}{|x|}
  \left(\begin{bmatrix}
        \beta_{11}(x) & \beta_{12}(x)\\
        \beta_{21}(x) & \beta_{22}(x)
    \end{bmatrix}   
    \begin{bmatrix}
        1 \\
        1
    \end{bmatrix} \right)
    \cdot \begin{bmatrix}
         R^{w}_1(2\tau - |x|,x)\\
         R^{w}_2(2\tau -|x|,x )
    \end{bmatrix} \,dx\\
    & \quad  + \frac{1}{4\pi}\int\limits_{|x-e|+|x| \leq 2\tau} \frac{1}{|x-e|}
  \left(\begin{bmatrix}
        \beta_{11}(x) & \beta_{12}(x)\\
        \beta_{21}(x) & \beta_{22}(x)
    \end{bmatrix}   
    \begin{bmatrix}
      R^{u^{(2)}}_1(2\tau - |x-e|,x)\\
         R^{u^{(2)}}_2(2\tau -|x-e|,x )
    \end{bmatrix} \right)
    \cdot \begin{bmatrix}
       1 \\
       1
    \end{bmatrix} \,dx \\
    & \quad  + \int\limits_{|x-e|+|x| \leq2\tau}\int\limits_{|x|}^{2\tau - |x-e|}
 \left( \begin{bmatrix}
        \beta_{11}(x) & \beta_{12}(x)\\
        \beta_{21}(x) & \beta_{22}(x)
    \end{bmatrix}   
    \begin{bmatrix}
        R^{u^{(2)}}_1(t,x)\\
         R^{u^{(2)}}_2(t,x)
    \end{bmatrix} \right)
    \cdot \begin{bmatrix}
         R^{w}_1(2\tau - t,x)\\
         R^{w}_2(2\tau -t,x )
    \end{bmatrix}  \,dt \,dx.
\end{align*}
After simplification, and using
\begin{align*}
    \lvert \nabla_x\left( 2\tau-|x|-|x-e| \right)\rvert &= \left| \frac{x}{\lvert x \rvert} + \frac{x-e}{\lvert x-e \rvert} \right|\\ 
    &= \left| \frac{|x-e|x +(x-e)|x|}{|x||x-e|} \right|\\
    &= \left| \frac{2\tau x - e|x|}{|x||x-e|} \right|
\end{align*}
together with the fact that  $\tau \in \left[0,\dfrac{T}{2}\right]$ with $1<T<\infty $ and the boundedness of $\overrightarrow{R^{w}}$ and $\overrightarrow{R}^{{u}^{(2)}}$ on compact subsets, we have 
\begin{align}\label{e21}
  \int\limits_{|x-e|+|x| = 2\tau} \frac{1}{\left|2\tau x - e|x|\right|}\left( \sum\limits_{i,j = 1}^{2}\beta_{ij}(x) \right)\,dS_x 
  \leq K \int\limits_{|x-e|+|x|  \leq 2\tau}  \frac{1}{|x-e||x|}\left( \sum\limits_{i,j = 1}^{2}\beta_{ij}(x) \right) \,dx
\end{align}
holds for each $\tau\in \left[0,\frac{T}{2}\right]$.
Now denote 
\begin{align}\label{e22}
    I(2\tau) := \int\limits_{|x-e|+|x| = 2\tau} \frac{1}{\left|2\tau x - e|x|\right|}\left( \sum\limits_{i,j = 1}^{2}\beta_{ij}(x) \right)\,dS_x.
\end{align}
By using the the prolate-spheroidal co-ordinates from Equation  \eqref{e13} and value of $\cosh{\rho} = 2\tau$, we get
\begin{align*}
    |x|= \left(x_1^2+x_2^2+x_3^2\right)^{1/2} = \frac{1}{2}\left(\cos^2{\phi}+\cosh^2{\rho}+2\cosh{\rho} \cos{\phi}\right)^{1/2} = \frac{1}{2}\left(2\tau + \cos{\phi} \right)
\end{align*} and 
\begin{align}\label{eq; simplify term}
    {\left|2\tau x - e|x|\right|} =|\left( 2\tau x_1 - |x|, 2\tau x_2, 2\tau x_3\right)| = \frac{1}{2}\sqrt{(4\tau^2 - \cos^2{\phi})(4\tau^2 - 1)}.
\end{align}
Combining the above expression together with the parametrizations for the surface measure given  by Equation \eqref{e14} in the integral \eqref{e22}, we obtain
\begin{align*}
    I(2\tau) = \frac{1}{8} \int\limits_{0}^{\pi}\int\limits_{0}^{2\pi} \sum\limits_{i,j = 1}^{2}\beta_{ij}(\rho, \theta, \phi)\left[\frac{\sinh{\rho}\sin{\phi}\sqrt{\cosh^{2}{\rho}-\cos^{2}{\phi}}}{\sqrt{(4\tau^2 - \cos^2{\phi})(4\tau^2 - 1)}}\right]\,d\theta \,d\phi,
\end{align*}
where 
\begin{align*}
    \beta_{ij}(\rho, \theta, \phi) = \beta_{ij}\left(   \frac{1}{2} + \frac{1}{2}\cosh{\rho}\cos{\phi}, \frac{1}{2} \sinh{\rho}\sin{\theta}\sin{\phi}, \frac{1}{2} \sinh{\rho}\cos{\theta}\sin{\phi} \right).
\end{align*}
After substituting the value $\cosh{\rho} = 2\tau$, $\sinh{\rho} = \sqrt{4\tau^{2}-1}$ and $\rho = \sinh^{-1}{\left(4\tau^{2}-1\right)} = \ln\left({2\tau + \sqrt{4\tau^{2}-1}}\right)$ in the above integral, we have
\begin{align}\label{e23}
     I(2\tau) =  \int\limits_{0}^{\pi}\int\limits_{0}^{2\pi} \sum\limits_{i,j = 1}^{2}\beta_{ij}\left(\ln{\left(2\tau + \sqrt{4\tau^{2}-1}\right)}, \theta, \phi\right)\sin{\phi}\,d\theta \,d\phi.
\end{align}
After substituting
$|x|=\tfrac{1}{2}(2\tau+\cos\phi)$, 
$|x-e|=\tfrac{1}{2}(2\tau-\cos\phi),$
and using Equations \eqref{e13} and   \eqref{e15} in the  volume integral appearing in the right-hand side of Equation \eqref{e21}, we get 
\begin{align*}
    \int\limits_{|x-e|+|x|\leq 2\tau}  \frac{1}{|x-e||x|}\left( \sum\limits_{i,j = 1}^{2}\beta_{ij}(x) \right) \,dx= \frac{1}{2}\int\limits_{\cosh{\rho}\leq 2\tau}\int\limits_{0}^{\pi}\int\limits_{0}^{2\pi} \sum\limits_{i,j = 1}^{2}\beta_{ij}(\rho, \theta, \phi)\sinh{\rho}\sin{\phi}\,d\theta\,d\phi\,d\rho.
\end{align*}
 Now use  $\cosh{\rho} = r$ and $ \rho = \ln\left( r + \sqrt{r^{2} - 1} \right)$ in the above expression, to arrive at  
\begin{align*}
    \int\limits_{|x-e|+|x|\leq 2\tau}  \frac{1}{|x-e||x|}\left( \sum\limits_{i,j = 1}^{2}\beta_{ij}(x) \right) \,dx = \frac{1}{2}\int\limits_{1}^{2\tau}\int\limits_{0}^{\pi}\int\limits_{0}^{2\pi} \sum\limits_{i,j = 1}^{2}\beta_{ij}( \ln\left( r + \sqrt{r^{2} - 1} \right), \theta, \phi)\sin{\phi}\,d\theta\,d\phi\,dr.
\end{align*}
Finally,  using Equation  \eqref{e23}, we get 
\begin{align*}
    \int\limits_{|x-e|+|x|\leq 2\tau}  \frac{1}{|x-e||x|}\left( \sum\limits_{i,j = 1}^{2}\beta_{ij}(x) \right) \,dx \leq C \int\limits_{1}^{2\tau} I(r)\,dr.
\end{align*}
Using the above inequality in the  estimate \eqref{e21} and the non-negativity of the coefficients
$\beta_{ij}(x)$, $1\le i,j\le 2$, we deduce
\begin{align*}
     I(2\tau)  \leq KC \int\limits_{1}^{2\tau} I(r)\,dr.
\end{align*}
Using Gr\"onwall's inequality, we have
\begin{align*}
    I(2\tau) = 0, ~\text{for}~ \tau \in \left[\frac{1}{2},\frac{T}{2}\right].
\end{align*}
From Equation \eqref{e22}, and using $\beta_{ij}(x)\geq 0$ for $1\leq i,j\leq 2$, we have $\beta_{ij}(x) = 0$, $1\leq i,j\leq 2$, and for each $x\in \mathbb{R}^3$.
Hence we have $\mathfrak{P}^{(1)}(x) = \mathfrak{P}^{(2)}(x) $ for each $x\in \mathbb{R}^3$ satisfying the condition $|x| + |x-e| \leq T$. This concludes the proof. \qed

\subsection{Proof of Theorems \ref{thm2} and \ref{thm4}}\label{Sec; proofs radial results}
In this subsection, we provide the proof of Theorem \ref{thm2} and Theorem \ref{thm4}, both of which are established under the symmetry assumptions on
the components of the matrices.

\subsubsection{Proof of Theorem~(\ref{thm2})}

Let's start by taking $\mathfrak{P}^{(k)}\in \mathcal{A}_1$, $k =1,2$.  Then the difference matrix $\mathfrak{P}(x) := \mathfrak{P}^{(1)}(x) - \mathfrak{P}^{(2)}(x)$ has the following form
\begin{align*}
  \mathfrak{P}(x)=
\begin{bmatrix}
\beta(x) & \beta(x)\\
\beta(x) & \beta(x)
\end{bmatrix}, 
\quad \mbox{where} \ \  \beta(x)=b(|x|),  
\end{align*}
where $b$ is a radial function. Consequently, we have
\begin{align*}
 \sum\limits_{i,j = 1}^{2}\beta_{ij}(x) =4b(|x|).  
\end{align*}
We begin by examining the surface integral appearing in \eqref{e11}, which we denote by
\begin{align*}
 I(\tau):=\int_{|x|=\tau}\frac{1}{|x|^{2}}
\left( \sum\limits_{i,j = 1}^{2}\beta_{ij}(x) \right)\,dS_x .  
\end{align*}
Using the spherical coordinates, we compute
\begin{align}\label{eq41}
I(\tau)
=4\int_{0}^{\pi}\int_{0}^{2\pi} b(\tau)\sin\phi\,d\theta\,d\phi
=8\pi\,b(\tau).
\end{align}
Next, consider the volume integral
\begin{align*}
    \int_{|x|\le \tau}\frac{1}{|x|}
\sum_{i,j=1}^{2}\beta_{ij}(x)\,dx.
\end{align*}
Using the radial structure of $\beta$ and spherical coordinates again, we obtain the following expression
\begin{align*}
\int_{|x|\le \tau}\frac{4\beta(x)}{|x|}\,dx
=4\int_{0}^{\tau}\int_{0}^{\pi}\int_{0}^{2\pi}
r\,b(r)\sin\phi\,d\theta\,d\phi\,dr\le C\int_{1}^{\tau}|b(r)|\,dr,
\end{align*}
where $C$ is some positive constant which is independent of $\tau$.

\noindent Combining this estimate with integral \eqref{eq41}, we arrive at the inequality
\begin{align}\label{estimate for b}
  |b(\tau)|\le C\int_{1}^{\tau}|b(r)|\,dr.  
\end{align}
Using the Gr\"onwall’s inequality to estimate \eqref{estimate for b}, we conclude that
\begin{align*}
  b(\tau)=0,
\qquad \tau\in\left[0,\frac{T}{2}\right].  
\end{align*}
Hence, the radial function $b$ vanishes on the interval $\left[0,\frac{T}{2}\right]$, which yields
\begin{align*}
  \mathfrak{P}(x)=0, \qquad \text{for each } x\in\mathbb{R}^3 \text{ with } |x|\le \frac{T}{2}.  
\end{align*}
We now consider the cases in which $\mathfrak{P}^{(k)}\in\mathcal{A}_2$, $k=1,2$.
In this case, the matrix-valued potential $\mathfrak{P}$ takes the form
\begin{align*}
  \mathfrak{P}(x)=
\begin{bmatrix}
\beta(x) & 0\\
0 & \beta(x)
\end{bmatrix},
\qquad \beta(x)=b(|x|). 
\end{align*}
Similarly, when $\mathfrak{P}^{(k)}\in\mathcal{A}_3$, $k=1,2$, the potential $\mathfrak{P}$ is given by
\begin{align*}
  \mathfrak{P}(x)=
\begin{bmatrix}
0 & \beta(x)\\
\beta(x) & 0
\end{bmatrix},
\qquad \beta(x)=b(|x|). 
\end{align*}
In both cases, an argument analogous to the one presented above yields the recovery of
the matrix potential $\mathfrak{P}$.
This concludes the proof.\qed

\subsubsection{Proof of Theorem~ (\ref{thm4})}

We begin by considering the surface integral appearing in Equation \eqref{e22}, which we denote by $I(2\tau)$. It is given by
\begin{align*}
    I(2\tau) := \int\limits_{|x-e|+|x| = 2\tau} \frac{1}{\left|2\tau x - e|x|\right|}\sum_{i,j=1}^{2} \beta_{ij}(x)\,dS_x.
\end{align*}
Let's start by taking $\mathfrak{P}^{(k)}\in \mathcal{A}_1$, $k =1,2$.  Then the difference matrix $\mathfrak{P}(x) := \mathfrak{P}^{(1)}(x) - \mathfrak{P}^{(2)}(x)$ has the following form
\begin{align*}
\mathfrak{P}(x) =
   \begin{bmatrix}
        \beta(x) & \beta(x)\\
        \beta(x) & \beta(x)
    \end{bmatrix}, \qquad \beta(x) = b(|x| +|e-x|), 
\end{align*}
 where $b$ satisfies the ellipsoidal symmetry, thus the above integrand depends only on the parameter $|x|+|x-e|$. Consequently, on the surface
$\{|x|+|x-e|=2\tau\}$, we have $ \beta(x) = b(2\tau)$. Using prolate spheroidal coordinates \eqref{e13}, together with
the standard surface measure \eqref{e14}   and Equation \eqref{eq; simplify term}, we obtain
\begin{align}\label{e41}
     I(2\tau) =  \int\limits_{0}^{\pi}\int\limits_{0}^{2\pi} b(2\tau)\sin{\phi}\,d\theta \,d\phi = 8\pi b(2\tau).
\end{align}
Next, we estimate the corresponding volume integral
\begin{align*}
    \int\limits_{|x-e|+|x|\leq 2\tau}  \frac{1}{|x-e||x|}\left( \sum\limits_{i,j = 1}^{2}\beta_{ij}(x) \right) \,dx. 
\end{align*}
 Again, using the prolate-spheroidal co-ordinates \eqref{e13} and volume element \eqref{e15} in the above integral, we have 
\begin{align*}
    \int\limits_{|x-e|+|x|\leq 2\tau}  \frac{1}{|x-e||x|}\left( \sum\limits_{i,j = 1}^{2}\beta_{ij}(x) \right) \,dx = \frac{1}{2}\int\limits_{\cosh{\rho}\leq 2\tau}\int\limits_{0}^{\pi}\int\limits_{0}^{2\pi} \sum\limits_{i,j = 1}^{2}\beta_{ij}(\rho, \theta, \phi)\sinh{\rho}\sin{\phi}\,d\theta\,d\phi\,d\rho.
\end{align*}
Using the representation $\beta_{ij}(x)=b_{ij}(|x|+|x-e|)$ and changing variables $\cosh{\rho} = r$ and $ \rho = \ln\left( r + \sqrt{r^{2} - 1} \right)$, we obtain
\begin{align}\label{eq; estimate with radial}
    \int\limits_{|x-e|+|x|\leq 2\tau}  \frac{4\beta(x)}{|x-e||x|} \,dx = 2\int\limits_{1}^{2\tau}\int\limits_{0}^{\pi}\int\limits_{0}^{2\pi} b(r)\sin{\phi}\,d\theta\,d\phi\,dr
    \leq C\int\limits_{1}^{2\tau}|b(r)|\,dr.
\end{align}
Combining the above estimate \eqref{eq; estimate with radial} with integral \eqref{e41}, we get the following inequality
\begin{align}\label{e42}
    |b(2\tau)| \leq C \int\limits_{1}^{2\tau}|b(r)|\,dr.
\end{align}
Using the Gr\"onwall’s inequality to the estimate \eqref{e42}, we conclude that 
\begin{align*}
    b(2\tau) = 0,~ \text{for}~ \tau \in \Big[\frac{1}{2},\frac{T}{2}\Big].
\end{align*}
Hence, it follows that $\mathfrak{P}(x)= 0$, for each $x\in \mathbb{R}^3$ satisfying the condition $|x| +|e-x| \leq T$.

We now consider the cases in which $\mathfrak{P}^{(k)}\in\mathcal{A}_2$, $k=1,2$.
In this case, the matrix-valued potential $\mathfrak{P}$ takes the form
\begin{align*}
  \mathfrak{P}(x)=
\begin{bmatrix}
\beta(x) & 0\\
0 & \beta(x)
\end{bmatrix},
\qquad \beta(x)=b(|x| +|e-x|). 
\end{align*}
Similarly, when $\mathfrak{P}^{(k)}\in\mathcal{A}_3$, $k=1,2$, the potential $\mathfrak{P}$ is given by
\begin{align*}
  \mathfrak{P}(x)=
\begin{bmatrix}
0 & \beta(x)\\
\beta(x) & 0
\end{bmatrix},
\qquad \beta(x)=b(|x| +|e-x|). 
\end{align*}
In both cases, an argument analogous to the one presented above yields the recovery of
the matrix potential $\mathfrak{P}$.
This concludes the proof.\qed

\section*{Acknowledgements}\label{sec:acknowledgements}
\begin{itemize}
      \item Rahul Bhardwaj acknowledges with gratitude the financial support received from the University Grants Commission (UGC), Government of India.
	\item Manmohan Vashisth is supported by the ISIRD project (No. 9--551/2023/IITRPR/10229) at IIT Ropar.
    \item This research also received partial support under the FIST program of the Department of Science and Technology, Government of India (Ref. No. SR/FST/MS-I/2018/22(C)).
    \item  The authors would like to sincerely thank Prof. Rakesh for insightful discussions and valuable suggestions, which significantly improved the quality of this work.\\
\end{itemize}

 \noindent \textbf{Data availability statement.} \ No datasets were generated or analyzed during the current study; therefore, data sharing is not applicable.

    \vspace{.1cm}
\noindent \textbf{Conflict of interest.} \
The authors declare that they have no conflicts of interest regarding the research, authorship, and/or publication of this article.
\bibliography{References1}

@book{Friedlander_1975,
  title={The wave equation on a curved space-time},
  author={Friedlander, F. G.},
  year={1975},
  publisher={Cambridge University Press},
  volume={2},
  series={Cambridge Monographs on Mathematical Physics}
}

@article{w1, author={M. Vashisth},
  title={An inverse problem for the wave equation with source and receiver at distinct points},
  journal={Journal of Inverse and Ill-posed Problems},
  volume={27},
  number={6},
  pages={835--843},
  year={2019}
}

@article{symes2009seismic,
  title={The seismic reflection inverse problem},
  author={Symes, W. W.},
  journal={Inverse problems},
  volume={25},
  number={12},
  pages={123008},
  year={2009},
  publisher={IOP Publishing}
}

@article {Romanov1992,
    AUTHOR = {Romanov, V. G.},
     TITLE = {On the problem of determining the coefficients in the lowest
              order terms of a hyperbolic equation},
   JOURNAL = {Sibirskij Matematicheskij Zhurnal},
  FJOURNAL = {Sibirskij Matematicheskij Zhurnal},
    VOLUME = {33},
      YEAR = {1992},
    NUMBER = {3},
     PAGES = {156--160, 220},
      ISSN = {0037-4474},
   MRCLASS = {35R30 (35L05)},
  MRNUMBER = {1177722},
MRREVIEWER = {Carlos E. Kenig},
       DOI = {10.1007/BF01201359},
       URL = {https://doi.org/10.1007/BF01201359}
}

@article {Rakesh2008,
    AUTHOR = {Rakesh},
     TITLE = {Inverse problems for the wave equation with a single
              coincident source-receiver pair},
   JOURNAL = {Inverse Problems},
  FJOURNAL = {Inverse Problems. An International Journal on the Theory and
              Practice of Inverse Problems, Inverse Methods and Computerized
              Inversion of Data},
    VOLUME = {24},
      YEAR = {2008},
    NUMBER = {1},
     PAGES = {015012, 16},
      ISSN = {0266-5611},
   MRCLASS = {35R30 (35L05)},
  MRNUMBER = {2389574},
       DOI = {10.1088/0266-5611/24/1/015012},
       URL = {https://doi.org/10.1088/0266-5611/24/1/015012}
}

@article {RakeshSacks2011,
    AUTHOR = {Rakesh and Sacks, P.},
     TITLE = {Uniqueness for a hyperbolic inverse problem with angular
              control on the coefficients},
   JOURNAL = {Journal of Inverse and Ill-Posed Problems},
  FJOURNAL = {Journal of Inverse and Ill-Posed Problems},
    VOLUME = {19},
      YEAR = {2011},
    NUMBER = {1},
     PAGES = {107--126},
      ISSN = {0928-0219},
   MRCLASS = {35R30 (35L05)},
  MRNUMBER = {2764014},
       DOI = {10.1515/JIIP.2010.059},
       URL = {https://doi.org/10.1515/JIIP.2010.059}
}

@article {Rakesh1998,
    AUTHOR = {Rakesh},
     TITLE = {Inversion of spherically symmetric potentials from boundary
              data for the wave equation},
   JOURNAL = {Inverse Problems},
  FJOURNAL = {Inverse Problems. An International Journal on the Theory and
              Practice of Inverse Problems, Inverse Methods and Computerized
              Inversion of Data},
    VOLUME = {14},
      YEAR = {1998},
    NUMBER = {4},
     PAGES = {999--1007},
      ISSN = {0266-5611},
   MRCLASS = {35R30 (35L05)},
  MRNUMBER = {1650634},
       DOI = {10.1088/0266-5611/14/4/011},
       URL = {https://doi.org/10.1088/0266-5611/14/4/011}
}

@article {Stefanov1990,
    AUTHOR = {Stefanov, P. D.},
     TITLE = {A uniqueness result for the inverse back-scattering problem},
   JOURNAL = {Inverse Problems},
  FJOURNAL = {Inverse Problems. An International Journal on the Theory and
              Practice of Inverse Problems, Inverse Methods and Computerized
              Inversion of Data},
    VOLUME = {6},
      YEAR = {1990},
    NUMBER = {6},
     PAGES = {1055--1064},
      ISSN = {0266-5611},
   MRCLASS = {35R30 (35L05)},
  MRNUMBER = {1094425},
       DOI = {10.1088/0266-5611/6/6/005},
       URL = {https://doi.org/10.1088/0266-5611/6/6/005}
}

@article{klibanov2005some,
  title={Some inverse problems with a ‘partial’ point source},
  author={Klibanov, M. V.},
  journal={Inverse problems},
  volume={21},
  number={4},
  pages={1379},
  year={2005},
  publisher={IOP Publishing}
}

@article{li2006estimation,
  title={Estimation of coefficients in a hyperbolic equation with impulsive inputs.},
  author={Li, S.},
  journal={Journal of Inverse \& Ill-Posed Problems},
  volume={14},
  number={9},
  year={2006}
}

@article{rakesh1993inverse,
  title={An inverse impedance transmission problem for the wave equation},
  author={Rakesh},
  journal={Communications in Partial Differential Equations},
  volume={18},
  number={3-4},
  pages={583--600},
  year={1993},
  publisher={Taylor \& Francis}
}

@article{rakesh2003inverse,
  title={An inverse problem for a layered medium with a point source},
  author={Rakesh},
  journal={Inverse Problems},
  volume={19},
  number={3},
  pages={497--506},
  year={2003},
  publisher={IOP PUBLISHING LTD TEMPLE CIRCUS, TEMPLE WAY, BRISTOL BS1 6BE, ENGLAND}
}

@article{sacks1996impedance,
  title={Impedance inversion from transmission data for the wave equation},
  author={Rakesh and Sacks, P.},
  journal={Wave Motion},
  volume={24},
  number={3},
  pages={263--274},
  year={1996},
  publisher={Elsevier}
}

@book{romanov2013integral,
  title={Integral geometry and inverse problems for hyperbolic equations},
  author={Romanov, V. G.},
  volume={26},
  year={2013},
  publisher={Springer Science \& Business Media}
}

@article{blaasten2017well,
  title={Well-posedness of the {G}oursat problem and stability for point source inverse backscattering},
  author={Bl{\aa}sten, E.},
  journal={Inverse Problems},
  volume={33},
  number={12},
  pages={125003},
  year={2017},
  publisher={IOP Publishing}
}

@article{ben2025stable,
  title={Stable recovery of a time dependent matrix potential for wave equation from arbitrary measurements},
  author={B. Fraj, O. and Rassas, I.},
  journal={Inverse Problems},
  year={2025}
}

@article{KRISHNAN2023622,
title = {Point sources and stability for an inverse problem for a hyperbolic {PDE} with space and time dependent coefficients},
journal = {Journal of Differential Equations},
volume = {342},
pages = {622-665},
year = {2023},
issn = {0022-0396},
doi = {https://doi.org/10.1016/j.jde.2022.10.025},
url = {https://www.sciencedirect.com/science/article/pii/S0022039622005988},
author = {V. P. Krishnan and  Rakesh and S. Senapati},
}

@article{Mishra10122021,
author = {R. K. Mishra and M. Vashisth},
title = {Determining the time-dependent matrix potential in a wave equation from partial boundary data},
journal = {Applicable Analysis},
volume = {100},
number = {16},
pages = {3492--3508},
year = {2021},
publisher = {Taylor \& Francis},
doi = {10.1080/00036811.2020.1721476},

URL = { https://doi.org/10.1080/00036811.2020.1721476},
eprint = {  https://doi.org/10.1080/00036811.2020.1721476}
}

@article{avdonin1992boundary,
  title={Boundary control and a matrix inverse problem for the equation},
  author={Avdonin, S. and Belishev, M. and Ivanov, S.},
  journal={Mathematics of the USSR-Sbornik},
  volume={72},
  number={2},
  pages={287},
  year={1992},
  publisher={IOP Publishing}
}

@article{khanfer2019inverse,
  title={Inverse problem for one-dimensional wave equation with matrix potential},
  author={Khanfer, A. and Bukhgeim, A.},
  journal={Journal of Inverse and Ill-posed Problems},
  volume={27},
  number={2},
  pages={217--223},
  year={2019},
  publisher={De Gruyter}
}

@inproceedings{eskin1997inverse,
  title={Inverse scattering problems for the {S}chr{\"o}dinger operators with external {Y}ang-{M}ills potentials},
  author={Eskin, G. and Ralston, J.},
  booktitle={CRM Proceedings and Lecture Notes},
  volume={12},
  pages={91--106},
  year={1997}
}

@article{bellassoued2019stability,
  title={Stability estimate in the determination of a time-dependent coefficient for hyperbolic equation by partial {D}irichlet-to-{N}eumann map},
  author={Bellassoued, M. and Rassas, I.},
  journal={Applicable Analysis},
  volume={98},
  number={15},
  pages={2751--2782},
  year={2019},
  publisher={Taylor \& Francis}
}

@inproceedings{kian2017unique,
  title={Unique determination of a time-dependent potential for wave equations from partial data},
  author={Kian, Y.},
  booktitle={Annales de l'Institut Henri Poincar{\'e} C, Analyse non lin{\'e}aire},
  volume={34},
  pages={973--990},
  year={2017},
  organization={Elsevier}
}

@article{vashisth2025unique,
  title={Unique determination of the damping coefficient in the wave equation using point source and receiver data},
  author={Vashisth, M.},
  journal={Proceedings-Mathematical Sciences},
  volume={135},
  number={1},
  pages={4},
  year={2025},
  publisher={Springer}
}

@article{rakesh2010stability,
  title={Stability for an inverse problem for a two-speed hyperbolic PDE in one space dimension},
  author={Rakesh and Sacks, P.},
  journal={Inverse Problems},
  volume={26},
  number={2},
  pages={025005},
  year={2010}
}

@article{Kumar2024StableDO,
  title={Stable determination of a time-dependent matrix potential for a wave equation in an infinite waveguide},
  author={N. Kumar and T. Sarkar and M. Vashisth},
  journal={Communications on Analysis and Computation},
  year={2024},
  url={https://api.semanticscholar.org/CorpusID:272368405}
}

@misc{filippas2025recoverymatrixvaluedpotential,
  title        = {Recovery of a matrix valued potential for the wave equation on stationary spacetimes},
  author       = {S. Filippas and L. Oksanen and M. Sarkkinen},
  eprint       = {2510.13410},
  archivePrefix= {arXiv},
  primaryClass = {math.AP},
  note         = {arXiv:2510.13410, 2025},
  url          = {https://arxiv.org/abs/2510.13410}
}

@incollection{uhlmann2014point,
  author       = {Rakesh and G. Uhlmann},
  title        = {The point source inverse back‐scattering problem},
  booktitle    = {Analysis, Complex Geometry, and Mathematical Physics: In Honor of Duong H. Phong},
  series       = {Contemporary Mathematics},
  volume       = {644},
  year         = {2015},
  publisher    = {American Mathematical Society},
  address      = {Providence, RI},
  doi          = {10.1090/conm/644},
}

@article{bube1983one,
  title={The one-dimensional inverse problem of reflection seismology},
  author={Bube, K. P. and Burridge, R.},
  journal={SIAM review},
  volume={25},
  number={4},
  pages={497--559},
  year={1983},
  publisher={SIAM}
}
\bibliographystyle{alpha}
\end{document}